\documentclass[11pt]{article}
\usepackage[margin=2cm]{geometry}                

\usepackage{hyperref}
\hypersetup{
    colorlinks=true,
    allcolors=violet
}

\usepackage{lmodern}

\usepackage{lscape}
\usepackage{graphicx}
\usepackage{amssymb}
\usepackage{amsmath}
\usepackage{amsthm}
\usepackage{epstopdf}
\usepackage{youngtab}
\usepackage{bm}
\usepackage{mathdots}
\usepackage{tikz}
\usetikzlibrary{calc}
\usetikzlibrary{patterns}
\usepackage{url}
\usepackage{array}
\begingroup
    \makeatletter
    \@for\theoremstyle:=definition,remark,plain\do{%
        \expandafter\g@addto@macro\csname th@\theoremstyle\endcsname{%
            \addtolength\thm@preskip\parskip
            }%
        }
\endgroup

\allowdisplaybreaks

\theoremstyle{theorem}
\newtheorem{thm}{Theorem}[section]
\newtheorem{prop}[thm]{Proposition}

\newtheorem{lem}[thm]{Lemma}

\theoremstyle{definition}
\newtheorem{defn}[thm]{Definition}

\newtheorem{rmk}[thm]{Remark}

\newcommand{\trop}[1]{{#1}^{\text{trop}}}

\newcommand{\Aa}{\mathbb{A}}
\newcommand{\CC}{\mathbb{C}}

\newcommand{\PP}{\mathbb{P}}
\newcommand{\QQ}{\mathbb{Q}}
\newcommand{\RR}{\mathbb{R}}
\newcommand{\TT}{\mathbb{T}}
\newcommand{\VV}{\mathbb{V}}
\newcommand{\ZZ}{\mathbb{Z}}

\newcommand{\sO}{\mathcal{O}}

\DeclareMathOperator{\Spec}{Spec}

\DeclareMathOperator{\Hom}{Hom}

\DeclareMathOperator{\SL}{SL}

\title{The tropicalisation of a $(-2,0)$-flop}
\author{Tom Ducat}
\date{22 February 2022}               

\begin{document}

\maketitle

\begin{abstract}
As a standard example in toric geometry, the Atiyah flop of a $(-1,-1)$-curve in a smooth 3-fold can be described combinatorially in terms of the two possible triangulations of a square cone. The flop of $(-2,0)$-curve cannot be realised in terms of toric geometry. Nevertheless, we explain how to construct a cone $\sigma$ in an integral affine manifold with singularities associated to the singularity $(P\in Y)$ at the base of a $(-2,0)$-flop. The two sides of the flop can then be described combinatorially in terms of two different subdivisions of $\sigma$. As an interesting byproduct of our construction, we can build a singularity $(P^\star\in Y^\star)$ which is mirror to $(P\in Y)$.
\end{abstract}


\section{Introduction}


A \emph{flopping curve} in a smooth 3-fold $X$ is a smooth rational curve $C\cong \PP^1\subset X$ which admits a small contraction $f\colon (C\subset X) \to (P\in Y)$ and satisfies $K_X\cdot C=0$. We let $\phi\colon (C\subset X)\dashrightarrow (C'\subset X')$ denote the flop of $C$. By a theorem of Laufer, the normal bundle $\mathcal N_{C/X}$ can be	 one of the following three possibilities
\[ \sO_{C}(-1)\oplus\sO_{C}(-1), \quad \sO_{C}(-2)\oplus\sO_{C}, \quad \sO_{C}(-3)\oplus\sO_{C}(1) \]
and all three cases can occur. We call $\phi$ a \emph{$(-a,a-2)$-flop} and the flopping curve $C$ a \emph{$(-a,a-2)$-curve}, for $a=1,2,3$ accordingly. 

\subsection{The Atiyah flop} \label{sect!atiyah-flop}

The \emph{Atiyah flop} is a $(-1,-1)$-flop over the hypersurface singularity 
\[ (P\in Y) = \left( 0 \in \VV(x_1x_3 - x_2x_4) \subset \mathbb A^4_{x_1,x_2,x_3,x_4} \right) \] 
which is isomorphic to the affine cone over $\PP^1\times \PP^1$ in its Segre embedding. The two sides of the flop are given by blowing up either one of the Weil divisors $D_{12}=\VV(x_1,x_2)$ or $D_{23}=\VV(x_2,x_3)$ in $Y$. 

\paragraph{Toric picture.}
This flop has a very simple and appealing description in terms of toric geometry. Consider a 3-dimensional torus $\TT=N\otimes_\ZZ \CC^\times$ with cocharacter lattice $N \cong \ZZ^3$ and character lattice $M = \Hom(N,\ZZ)\cong \ZZ^3$. We consider the coordinates $x_1,x_2,x_3,x_4$ on $Y$ as monomials $x_i=u^{m_i}$ on $\TT$ corresponding to the vectors
\[ m_1 = (1,0,0), \: m_2 = (0,1,0), \: m_3 = (0,0,1), \: m_4 = (1,-1,1) \]
which span a square cone $\sigma^\star = \langle m_1,m_2,m_3,m_4\rangle$ in the vector space $M_\RR:=M\otimes \RR$. Then $Y = \Spec\CC[\sigma^\star\cap M]$ is the partial compactification of $\TT$ determined by the dual cone $\sigma = \langle n_{12},n_{23},n_{34},n_{41} \rangle$ in the dual space $N_\RR:=N\otimes \RR$, where the four generators
\[ n_{12} = (1,1,0), \: n_{23} = (0,1,1), \: n_{34} = (0,0,1), \: n_{41} = (1,0,0) \]
correspond to the four components $D_{ij} = \VV(x_i,x_j)$ of the boundary divisor $D=Y\setminus \TT$. The two small resolutions $X$ and $X'$ of $Y$ can now be obtained as the toric varieties associated to the fans obtained by subdividing $\sigma$ into two smooth triangular cones in one of the two possible ways; $\sigma = \sigma_{12}\cup\sigma_{34}$ or $\sigma = \sigma_{23}\cup\sigma_{41}$, where $\sigma_{ij}\subset \sigma$ is the subcone $\sigma_{ij} = \langle n_{kl} : kl\neq ij \rangle$. 
\begin{figure}[htbp]
\begin{center}
\begin{tikzpicture}[scale=1.5]
   \begin{scope} [xshift = 4cm]
   \draw[gray!50] (1/2,-1) -- (0,1) (1/2,-1) -- (1,1);
   \draw (1/2,-1) -- (0,0) (1/2,-1) -- (1,0);
   \draw (0,0) -- (1,0) -- (1,1) -- (0,1) -- cycle;
   \node at (0,0) {$\bullet$};
   \node at (1,0) {$\bullet$};
   \node at (1,1) {$\bullet$};
   \node at (0,1) {$\bullet$};
   \node at (2/7,5/7) {$\sigma_{41}$};
   \node at (5/7,2/7) {$\sigma_{23}$};
   \node at (1/2,-1) {\scriptsize $\bullet$};
   \draw (0,0) -- (-0.1,0.2);
   \draw (1,0) -- (1.1,0.2);
   \draw (0,1) -- (-0.05,1.2);
   \draw (1,1) -- (1.05,1.2);
   
   \draw (0,0) -- (1,1);
   \end{scope}
   
   \draw[<->,dashed] (2,0.5) to node[midway,above] {flop} (3,0.5);
   
   \begin{scope} [xshift = 0cm]
   \draw[gray!50] (1/2,-1) -- (0,1) (1/2,-1) -- (1,1);
   \draw (1/2,-1) -- (0,0) (1/2,-1) -- (1,0);
   \draw (0,0) -- (1,0) -- (1,1) -- (0,1) -- cycle;
   \node at (0,0) {$\bullet$};
   \node at (1,0) {$\bullet$};
   \node at (1,1) {$\bullet$};
   \node at (0,1) {$\bullet$};
   \node at (2/7,2/7) {$\sigma_{34}$};
   \node at (5/7,5/7) {$\sigma_{12}$};
   \node at (1/2,-1) {\scriptsize $\bullet$};
   \draw (0,0) -- (-0.1,0.2);
   \draw (1,0) -- (1.1,0.2);
   \draw (0,1) -- (-0.05,1.2);
   \draw (1,1) -- (1.05,1.2);
   
   \draw (0,1) -- (1,0);
   \end{scope}
\end{tikzpicture}
\caption{The combinatorial picture of the Atiyah flop.}
\label{fig!atiyah-flop}
\end{center}
\end{figure}
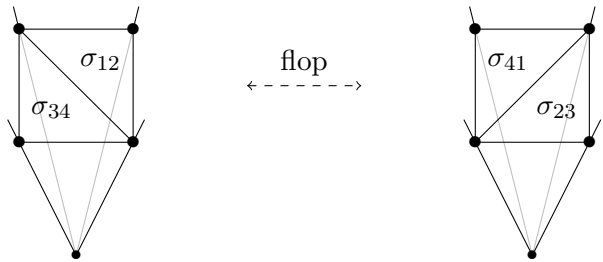
Indeed, the small resolution $X$ is given by gluing together the two smooth affine charts $X=\Spec\CC[\sigma_{12}^\star\cap M] \cup \Spec\CC[\sigma_{34}^\star\cap M]$ and similarly for $X'$.

\subsection{Other flops}

Since all $(-1,-1)$-flops are locally analytically isomorphic to the Atiyah flop, this construction gives a complete description of the local behaviour of any such flop. There are, however, no examples of toric $(-2,0)$-flops or $(-3,1)$-flops. Indeed, in these cases a torus-invariant curve in a smooth toric 3-fold $X$ with such a normal bundle corresponds to a union of two cones in the fan of $X$ which is not strictly convex. In this paper we will extend the combinatorial picture of Figure~\ref{fig!atiyah-flop} to the case of a $(-2,0)$-flop by allowing $\sigma$ to be a cone in an integral affine manifold with singularities. 

\paragraph{The classification of $(-2,0)$-flops.}
For a 3-fold flopping curve $C\subset X$, Reid \cite[Definition~5.3]{reid} introduced a discrete invariant 
\[ w(C) = \max\left\{ n : C\subset C_n\subset X  \text{ where } C_n\cong C\times \Spec\CC[t]/(t^n) \right\} \]
called the \emph{width} of $C$. The curve $C$ is a $(-1,-1)$-curve if and only if $w(C)=1$. Moreover, if $C$ is a $(-2,0)$-curve of width $w(C)=n$ then by \cite[(5.13.b)]{reid} the flop of $C$ is locally analytically isomorphic to one of a discrete family of flops $\phi_n\colon X\dashrightarrow X'$ over the hypersurface singularity 
\[ Y = \VV(xz - y(y+t^n)) \subset \Aa^4_{x,y,z,t}. \]
As with the Atiyah flop before, the two sides of the flop can be obtained by blowing up either one of the Weil divisors $D_1=\VV(x,y)$ or $D_2=\VV(y,z)$ inside $Y$. If we allow the case $n=1$ then we see that this construction includes the Atiyah flop, albeit presented in non-toric coordinates.


\subsection{Generalising toric geometry} \label{sect!generalising}

The toric construction of the Atiyah flop relies on identifying a dense open torus $\TT\subset Y$, after which the theory of toric geometry reduces the geometry of the space $Y$ to the combinatorial data of a cone $\sigma$ defined in the cocharacter space $N_\RR$ of $\TT$. The integral points of $\sigma$ correspond to torus-invariant divisorial valuations with centre on $Y$. Moreover, we have an associated dual cone $\sigma^\star$ in the character space $M_\RR$ of $\TT$, whose integral points correspond to the monomials in the ring $\CC[Y]$ forming a distinguished additive basis of $\CC[Y]$.   

\paragraph{Integral affine manifolds with singularities.}
The \emph{Gross--Siebert program} is concerned with understanding mirror symmetry through toric degenerations. Part of the technology developed in the course of this program 
allows one to generalise the machinery of toric geometry to study compactifications of a positive, maximal log Calabi--Yau variety~$U$. To do this we replace the real manifold $N_\RR$ with an \emph{integral affine manifold with singularities} $\trop U$, called the \emph{tropicalisation of $U$}. By definition, this is a real manifold $\trop U$ with a distinguished set of integral points $\trop U(\ZZ)\subset \trop U$ (playing the role of the lattice $N\subset N_\RR$) and a given subset $\Delta\subset \trop U$ of codimension at least 2, called the \emph{singular locus} of $\trop U$. The complementary open subset $\trop U \setminus \Delta$ now has an integral affine structure for which the set of integral points is given by $\trop U(\ZZ)\setminus \Delta$, however this integral affine structure cannot be extended to the whole of $\trop U$ in general. From the way in which $\trop U(\ZZ)$ is defined \cite[Definition~1.7]{ghk} the integral points of $\trop U$ correspond to divisorial valuations along boundary components in compactifications of $U$.

The analog of the character space $M_\RR$ in this situation is conjecturally given by the tropicalisation of a \emph{mirror} log Calabi--Yau variety $V=U^\star$ \cite[Conjecture~1.1]{hk}. This is a conjectural duality between deformation families of positive, maximal log Calabi--Yau varieties that extends the usual duality $\TT \leftrightarrow \TT^\star=M\otimes_\ZZ \CC^\times$ of tori, which comes from exchanging the roles of the two lattices $N$ and $M$. Thus the integral points $\trop V(\ZZ)\subset \trop V$ are expected to correspond to an certain additive basis of $\CC[U]$, called the basis of \emph{theta functions}, and vice-versa. One approach to understanding the mirror $V$ is to build the coordinate ring $\CC[V]$ directly from the tropicalisation $\trop U$, by equipping it with the structure of a scattering diagram.

\paragraph{Cluster varieties.}
In our case we identify a particularly simple affine log Calabi--Yau variety $U$ which is a Zariski open subset $U\subset Y$ in the base of a $(-2,0)$-flop. This $U$ is a \emph{cluster variety}, in the sense of \cite{ghk,ghkk,hk}, and can be described as a certain type of blowup of a toric variety. Indeed, continuing with notation $\TT$, $N$, $M$ etc., introduced above, then there is a pair $(n,m)\in N\times M$ such that $U$ is obtained by blowing up $(T,B)$, a toric compactification of $\TT$, along the curve $Z= \VV(1+u^m)$ in the boundary divisor $B_n\subset T$. 

Since $U$ is positive it has a mirror given by the \emph{Fock--Goncharov dual cluster variety} $V=U^\star$ \cite[Proposition~0.7]{ghkk}. In this case $V$ is obtained from the data of a pair of primitive integral vectors $(m,-n)\in M\times N$ by blowing up $(T^\star,B^\star)$, a toric compactification of $\TT^\star$, along the curve $Z^\star= \VV(1+v^{-n})$ in the boundary divisor $B^\star_m\subset T^\star$. In particular, since $U$ and $V$ are both given by blowing up a smooth irreducible reduced curve in the boundary of a toric 3-fold, by \cite[Remark~3.5]{hk} they are both isomorphic to $W\times \CC^\times$, where $W$ is the two-dimensional cluster variety obtained by the nontoric blowup of a single reduced point in the boundary of a toric surface.

The two spaces $\trop U$ and $\trop V$ are now dual to one another under a pairing $\langle{\cdot},{\cdot}\rangle\colon \trop U\times \trop V\to \RR$ which extends the pairing $\langle{\cdot},{\cdot}\rangle\colon \trop U(\ZZ)\times \trop V(\ZZ)\to \ZZ$ given by $\langle u,v \rangle=\nu_u(\vartheta_v) = \nu_v(\vartheta_u)$, where we can interpret $u\in \trop U(\ZZ)$ as a divisorial valuation $\nu_u\colon \CC(U)^\times\to \ZZ$ and $v\in \trop V(\ZZ)$ as a regular function $\vartheta_v\in \CC[U]$, or vice-versa.\footnote{In general the equality $\nu_u(\vartheta_v) = \nu_v(\vartheta_u)$ is a conjecture \cite[Remark~9.11]{ghkk} which has been proven by Mandel \cite{man} for two-dimensional cluster varieties. However, by the same remark in \cite{ghkk}, in our case it follows from the fact that the $U$ we work with has cluster complex equal to the whole of $\trop U$ (and similarly for $V$).} Using this dual pairing, Mandel \cite{man} has generalised many of the notions of toric geometry to this setting, particularly in the 2-dimensional case.

\subsection{Outline of the paper} \label{sect!strategy}

As a (fun) application of this theory, we study $(-2,0)$-flops in an analogous way to the toric treatment of the Atiyah flop, beginning by identifying a divisor $D\subset Y$ such that the complement is a cluster variety $U=Y\setminus D$. In \S\ref{sect!2} we then construct the tropicalisation $\trop U$ together with the convex cone $\sigma\subset \trop U$ generated by the integral points $d_i\in \trop U(\ZZ)$ that correspond to irreducible components $D_i\subset D$. We find that, for a unique choice of toric model for $U$ (which fixes the singular locus of $\trop U$), there are exactly two ways to subdivide $\sigma$. In Theorem~\ref{thm!20-flop} we show that these two subdivisions give the dual intersection complex of the pair $(X,\widetilde D)$, where $f\colon X\to Y$ is one side of the $(-2,0)$-flop and $\widetilde D$ is the strict transform of $D$ under $f$. Thus we interpret these two subdivisions as the combinatorial analogue of the two subdivisions of the toric cone shown in Figure~\ref{fig!atiyah-flop}.

In \S\ref{sect!pagoda} we show that we can resolve $\sigma\subset \trop U$ into smooth cones by subdividing along interior rays (the analogue of resolving a toric singularity by toric blowups). As seen in Theorem~\ref{thm!pagoda}, this recovers \emph{Reid's pagoda} -- a certain resolution of $(-2,0)$-flop considered by Reid \cite[\S6]{reid}.

In \S\ref{sect!dual-side} we construct the dual cone $\sigma^\star$ in the dual integral affine manifold $\trop V$ and describe the pairing between $\trop U$ and $\trop V$. For either one of our given subdivisions $\sigma=\tau_1\cup \tau_2$, precisely one of the two components ($\tau_1$ say) is a convex cone in $\trop U$ and we show that the affine variety $\Spec \CC[\tau_1^\star\cap \trop V(\ZZ)]$ gives one of the affine patches on one side of the flop. Unfortunately at this point the analogy with the toric picture breaks down, since we do not know how to recover the other affine patch of the flop this way.

Lastly, in \S\ref{sect!mirror} we note that as a consequence of construction, we can construct a singularity $(P^\star\in Y^\star)$ as a compactification of $V$, which is mirror to $(P\in Y)$. This mirror singularity is obtained by exchanging the role of the two cones $\sigma\subset \trop{U}$ and $\sigma^\star\subset \trop V$. Perhaps surprisingly, $Y^\star$ turns out to be a nonisolated affine 3-fold singularity. In Theorem~\ref{thm!mirror-thm} we give a presentation for the coordinate ring $\CC[Y^\star]$ and produce a crepant resolution of $(P^\star\in Y^\star)$. We end with a speculative question about the relationship between $Y$ and $Y^\star$.

\paragraph{Generalisations.}
Unfortunately the methods employed in this paper do not seem well-suited to the much more interesting case of $(-3,1)$-flops. In these cases a `general elephant' (i.e.\ a general anticanonical divisor) in the base of the flopping contraction $(P\in Y)$ has a Du Val singularity of type $D$ or $E$ (rather than type $A$). However, given the known connections between cluster algebras and Mori flips \cite{duc} it seems highly likely that one could study 3-dimensional flips and divisorial contractions with a type $A$ elephant from a similar point of view. We hope that there will also be applications in higher dimensions.

\paragraph{Acknowledgements.}
In the course of preparing this paper I had helpful conversations with Gavin Brown, Jonny Evans and Helge Ruddat, and I am grateful for the improvements suggested by the anonymous referee.

\section{The tropicalisation of the base of a $(-2,0)$-flop} \label{sect!2}

Recall that the base of a $(-2,0)$-flop is given by the affine hypersurface singularity 
\[ (P\in Y) = \left(0\in \VV(xz-y(y+t^n))\subset \Aa^4_{x,y,z,t}\right). \]
We want to construct the tropicalisation of $Y$ giving $\sigma\subset \trop U$, a cone in an integral affine manifold with singularities, according to the strategy outlined in \S\ref{sect!strategy}.

\subsection{The log Calabi--Yau variety $U$}

We first describe our choice of log Calabi--Yau open subset $U\subset Y$.

\begin{lem}\label{lem!resolution}
Consider the reduced effective divisor $D = \VV(yt)\subset Y$ with three irreducible components $D=\sum_{i=1}^3D_i$, given by
\[ D_1 = \VV(x,y), \quad D_2 = \VV(y,z) \quad \text{and} \quad D_3 = \VV(xz - y^2, t). \]
The complement $U=Y\setminus D$ is a log Calabi--Yau variety and $(Y,D)$ is a dlt partial compactification of $U$ with a maximal centre.
\end{lem}

\begin{proof}
We note that $U$ is a \emph{cluster variety} in the sense of \cite{ghk,ghkk}. In particular, $U$ contains two dense open torus charts $\TT_1=(\CC^\times)^3_{x,y,t}$ and $\TT_2=(\CC^\times)^3_{y,z,t}$ which are glued together by a \emph{mutation} $\mu\colon \TT_1\dashrightarrow \TT_2$, given by $\mu^*(y,z,t) = \left(\frac{y(y+t^n)}{z},y,t\right)$. The mutation is volume-preserving, in the sense that $\mu^*\Omega_{\TT_2} = \Omega_{\TT_1}$ for the volume forms $\Omega_{\TT_1} = \frac{dx}{x}\wedge \frac{dy}{y}\wedge \frac{dt}{t}$ and $\Omega_{\TT_2} = \frac{dy}{y}\wedge \frac{dz}{z}\wedge \frac{dt}{t}$ on the two torus charts. Since $\TT_1\cup \TT_2$ covers $U$ up to a set of codimension 2 and $U$ is smooth, these volume forms extend to give a nonvanishing volume form $\Omega_U\in H^0(U,K_U)$ on the whole of $U$. Thus $U$ is a log Calabi--Yau variety.

To see that $(Y,D)$ is a dlt pair with a maximal centre we can pass to the small resolution 
\[ f\colon X = \VV\left( y\alpha - z\beta, \: x\alpha - (y+t^n)\beta \right)\subset \Aa^4\times \PP^1_{\alpha:\beta} \to Y, \] 
and consider the pair $(X,\widetilde D)$, where $\widetilde D$ is the strict transform of $D$. For the two affine patches that cover $X$ we have 
\[ \left.\left(X,\widetilde D\right)\right|_{\alpha=1} \cong \left(\Aa^3_{\beta,z,t}, \; \VV(\beta zt)\right) \quad \text{and} \quad \left.\left(X,\widetilde D\right)\right|_{\beta=1} \cong \left(\Aa^3_{\alpha,x,t}, \; \VV(t(x\alpha - t^n))\right). \]
The boundary divisor in the first patch is the union of the three coordinate planes, whilst the boundary of the second patch is the union of a plane with an $A_{n-1}$ singularity, meeting transversely along their toric boundary strata. Both are well-known examples of 3-dimensional dlt pairs with a maximal centre.
\end{proof}

\subsection{The tropicalisation $U^{\text{trop}}$}

\paragraph{A toric model for $U$.}
To understand the tropicalisation of $U$ we begin by exhibiting a toric model for $U$. In other words, we give a description of $U$ in terms of the blowup of a toric variety. To fix some notation we let $N\cong \ZZ^3$ be the cocharacter lattice of the torus $\TT=(\CC^\times)^3_{x,y,t}$ and we let $M=\Hom(N,\ZZ)$ be the character lattice, where $x=u^{(1,0,0)}$, $y=u^{(0,1,0)}$ and $t=u^{(0,0,1)}$. In any toric compactification of $\TT$ we let $B_v$ denote the toric boundary divisor corresponding to a primitive vector $v\in N$.

\begin{lem}
Consider the toric variety $T=\Aa^1_x\times(\CC^\times)^2_{y,t}$ with dense torus $\TT$ and toric boundary divisor $B:=B_{(1,0,0)}=\VV(x)$. Let $Z$ be the curve $Z=\VV(x,y+t^n)\subset B\subset T$ and let $f\colon (\widetilde T,\widetilde B)\to (T,B)$ be the blowup of $T$ along $Z$, where $\widetilde B$ is the strict transform of $B$. Then $U$ is isomorphic to $\widetilde T\setminus \widetilde B$.
\end{lem}

\begin{proof}
Since $y\neq0$ on $U$, we can write $U=\left\{x\cdot \frac{z}{y} = y+t^n\right\}\subset \Aa^2_{x,z}\times(\CC^\times)^2_{y,t}$ and consider the projection $\pi\colon U\to T$. This is a birational map from $U$ onto $T$ with image $\pi(U) = \TT\cup Z$. This is an isomorphism over $\TT$ and which blows down the divisor $E=\VV(x,y+t^n)\subset U$ onto $Z$. From the equation defining $U$, we see that $U$ is isomorphic to a Zariski open subset of the blowup of $Z\subset T$.
\end{proof}

\paragraph{The tropicalisation $\trop U$.}
We can now use this description of $U$ to construct the tropicalisation $\trop{U}$, by redefining the integral affine structure on $\trop{\TT}= N_\RR\cong\RR^3$ according to this nontoric blowup. To do this consider the plane $H=(0,-1,n)^\perp\subset N_\RR$, the line $\ell=\RR(0,n,1)\subset H$, the open halfplanes $H^\pm=\{ h\in H : \pm\langle h,(1,0,0)\rangle > 0 \}$ and the closed halfplanes ${\overline H}^\pm=\{ h\in H : \pm\langle h,(1,0,0)\rangle \geq 0 \}$. 

\begin{rmk}\label{rmk!trop-not-unique}
We note that, despite the potentially misleading notation, the tropicalisation $\trop U$ is not canonically determined by $U$. Instead it depends on the choice of toric compactification $(T,B)$ used in the toric model for $U$. We will make one such choice in Proposition~\ref{prop!Utrop} and then describe how the integral affine structure on $\trop U$ changes as the choice of toric model changes in Remark~\ref{rmk!moving-ell}.
\end{rmk}

\begin{prop} \label{prop!Utrop}
The tropicalisation $\trop{U}$ (with respect to the toric compactification $(T,B)$ described in the proof below) is given by the integral affine manifold with singularities obtained by redefining the integral affine structure on $N_\RR\setminus \ell$ so that the tangent vector of a straight line bends by the matrix
\[ M_\ell = \begin{pmatrix}
1 & 1 & -n \\
0 & 1 & 0 \\
0 & 0 & 1
\end{pmatrix} \]
as it passes through $H^+$ in the direction of the normal vector $(0,-1,n)$. In particular the singular locus of $\trop{U}$ is given by the line $\ell$ and $M_\ell\in\SL(3,\ZZ)$ is the monodromy around $\ell$. Moreover we can identify the integral points $\trop U(\ZZ)\subset \trop U$ with the lattice $N\subset N_\RR$.
\end{prop}

\begin{proof}

Consider the toric variety $(T,B)$ obtained via the sequence of toric blowups $T\to (\PP^1)^3_{x,y,t}$ given by subdividing the standard fan for $(\PP^1)^3$ along the rays $\RR_{\geq0}(0,k,1)$ and $\RR_{\geq0}(0,-k,-1)$ for $k=1,\ldots,n$. The effect of this blowup on the first orthant of $N_\RR$ is shown in Figure~\ref{fig!toric-model} and it resolves $Z\subset T$ into a curve that meets the toric boundary $B\subset T$ transversely. Consider the torus-invariant curves $C_{\pm}=B_{(1,0,0)}\cap B_{\pm(0,n,1)}\cong \PP^1$. When we blow up $Z\subset T$, this changes the intersection number $\left(C_+|_{B_{(0,n,1)}}\right)^2$ from $0$ to $-1$ (depicted by the circled number in Figure~\ref{fig!toric-model}), and similarly for $C_-$, but leaves all other intersection numbers of curves in the boundary the same.
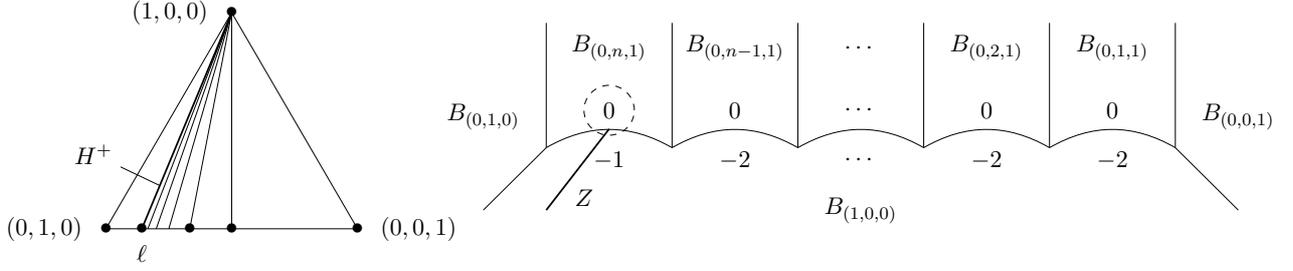
\begin{figure}[htbp]
\begin{center}
\resizebox{\textwidth}{!}{
\begin{tikzpicture}[scale=2]
    \draw (0,1) -- (0,0) -- (-0.5,-0.5);
    \draw (5,1) -- (5,0) -- (5.5,-0.5);
    \draw (1,1) -- (1,0) (2,1) -- (2,0) (3,1) -- (3,0) (4,1) -- (4,0);
    \draw (0,0) to [in=150,out=30] (1,0);
    \draw (1,0) to [in=150,out=30] (2,0);
    \draw (2,0) to [in=150,out=30] (3,0);
    \draw (3,0) to [in=150,out=30] (4,0);
    \draw (4,0) to [in=150,out=30] (5,0);
    \draw[thick] (0.5,0.15) -- (0,-0.5);
    \node at (0.5,0.3) {$0$};
    \node at (1.5,0.3) {$0$};
    \node at (2.5,0.3) {$\cdots$};
    \node at (3.5,0.3) {$0$};
    \node at (4.5,0.3) {$0$};
    \node at (0.5,-0.1) {$-1$};
    \node at (1.5,-0.1) {$-2$};
    \node at (2.5,-0.1) {$\cdots$};
    \node at (3.5,-0.1) {$-2$};
    \node at (4.5,-0.1) {$-2$};
    \node at (-0.5,0.25) {$B_{(0,1,0)}$};
    \node at (2.5,-0.5) {$B_{(1,0,0)}$};
    \node at (5.5,0.25) {$B_{(0,0,1)}$};
    \node at (0.5,0.8) {$B_{(0,n,1)}$};
    \node at (1.5,0.8) {$B_{(0,n-1,1)}$};
    \node at (2.5,0.8) {$\cdots$};
    \node at (3.5,0.8) {$B_{(0,2,1)}$};
    \node at (4.5,0.8) {$B_{(0,1,1)}$};
    \draw[dashed] (0.5,0.3) circle (0.2cm);
    \node at (0.3,-0.4) {$Z$};
    
    \begin{scope}[xshift = -3.5cm,yshift=-0.65cm]
    \draw (0,0) -- (2,0) -- (1,{sqrt(3)}) -- cycle;
    \draw (1,{sqrt(3)}) -- (1,0);
    \draw (1,{sqrt(3)}) -- (2/3,0);
    \draw (1,{sqrt(3)}) -- (1/2,0);
    \draw (1,{sqrt(3)}) -- (2/5,0);
    \draw (1,{sqrt(3)}) -- (2/6,0);
    \draw[thick] (1,{sqrt(3)}) -- (2/7,0);
    \node at (0,0) [label={left:$(0,1,0)$}]{$\bullet$};
    \node at (2/3,0) {$\bullet$};
    \node at (2/7,-0.2) {$\ell$};
    \node at (2/7,0) {$\bullet$};
    \node at (1,0) {$\bullet$};
    \node at (2,0) [label={right:$(0,0,1)$}]{$\bullet$};
    \node at (1,{sqrt(3)}) [label={left:$(1,0,0)$}]{$\bullet$};
    \draw (-0.1,0.6) -- (2/7 + 5*0.2/7, {0.2*sqrt(3)});
    \node[fill=white] at (-0.1,0.6) {$H^+$};
    \end{scope}
\end{tikzpicture}}
\caption{The fan of $(T,B)$ and the toric boundary of $(T,B)$.}
\label{fig!toric-model}
\end{center}
\end{figure}

In the fan of a 3-dimensional toric variety $W$, for any two smooth cones $\sigma_1=\langle v_1,v_2,v_3\rangle$ and $\sigma_2=\langle v_2,v_3,v_4\rangle$ which intersect in a face $\tau=\langle v_2,v_3\rangle$, the four primitive generators $v_1,v_2,v_3,v_4\in N$ satisfy a relation $v_4 = av_2+bv_3-v_1$, from which we read off that the torus invariant curve $C_\tau \subset W$ has normal bundle $\mathcal N_{C_\tau/W}=\sO(-a,-b)$. If $\pi\colon \widetilde W\to W$ is a nontoric blowup in the boundary of $W$ which changes the normal bundle of this curve $\widetilde C\subset \widetilde T$ to $\mathcal N_{\widetilde C_\tau/\widetilde W}=\sO(-c,-d)$, then we can obtain $\trop{\widetilde W}$ by redefining the integral affine structure on $N_\RR\setminus \left(\RR_{\geq0}v_2 \cup \RR_{\geq0}v_3\right)$. In particular, we require that tangent vectors are multiplied by a matrix $M\in\SL(3,\ZZ)$ as we pass through the wall $\tau$, from $\sigma_1$ into $\sigma_2$, so that we have $v_2=Mv_2$, $v_3=Mv_3$ and $v_4 = cv_2+dv_3-Mv_1$ (rather than the conventional rule above). In our case, we must redefine the affine structure as we pass through the two cones contained in the halfspace ${\overline H}^+\subset N_\RR$, and taking 
\[ v_1,\; v_2,\; v_3,\; v_4=(0,1,0),\; (1,0,0),\; (0,n,1),\; (0,n-1,1) \] 
and $c=d=1$ implies that $M=M_\ell$ is the matrix in the statement of the lemma. Apriori the singular locus of $\trop U$ is given by the ray $\rho=\RR_{\geq0}(1,0,0)$ and the line $\ell=\RR(0,n,1)$, but the integral affine structure extends over $\rho$.
\end{proof}

\paragraph{Moving the singular line of $\trop U$.}
As mentioned above in Remark~\ref{rmk!trop-not-unique}, changing the choice of toric model used in Propostion~\ref{prop!Utrop} will change the integral affine structure on $\trop U$ by moving the singular locus.

\begin{rmk}\label{rmk!moving-ell}
In particular, the singular line $\ell\subset \trop U$ can be moved to any line of the form $\ell = \RR(a,nb,b)$ for a pair of integers $a,b\in\ZZ_{>0}$. This is done by making the appropriate weighted blowup of the curves $C_\pm$ before making the nontoric blowup of $Z\subset T$. The difference now is that the `endpoints' of the curve $Z\subset B$ meet the boundary divisors $B_{\pm(a,nb,b)}$, rather than $B_{\pm(0,n,1)}$. Making this change in the integral affine structure on $\trop U$ does not change the underlying manifold $\trop U\cong \RR^3$, the set of integral points $\trop U(\ZZ) \cong \ZZ^3\subset \RR^3$ or the monodromy $M_\ell$ around $\ell$.

In fact, for reasons that will become clear after the discussion in the paragraph proceeding Figure~\ref{fig!bending} below, we will find it especially convenient to consider the case in which $a=nb$. Thus, from now on we consider the integral affine structure on $\trop U$ to be determined by fixing the choice of singular line $\ell = \RR(n,n,1)$.
\end{rmk}

\subsection{The cone $\sigma\subset U^{\text{trop}}$} \label{sect!cone}

\paragraph{Convexity in $\trop U$.}
We first explain what it means for a subset of $\trop U$ to be convex. 

\begin{defn}
A \emph{straight line segment} in $\trop U$ is a given by a map $\phi\colon [0,1]\to \trop U\setminus \ell$ which is linear when considered in the underlying affine structure on $\RR^3\setminus \overline H^+$ and whose tangent direction bends by $M_\ell^{
\pm1}$ whenever it passes through $H^+$ in the direction of $\pm(0,-1,n)$. A subset $S\subset \trop U$ is \emph{convex} with respect to the integral affine structure on $\trop U$ if any straight line segment that passes between any two points $p,q\in S$ is wholly contained in $S$. 
\end{defn}

This is a the most naive notion of convexity in $\trop U$ and unfortunately is not preserved under the kinds of changes in integral affine structure on $\trop U$ described in Remark~\ref{rmk!moving-ell}. Since we are now working with the fixed choice of singular line $\ell=\RR(n,n,1)$, unless otherwise stated convexity in $\trop U$ will always be assumed with respect to this integral affine structure.

There are more sophisticated versions of convexity which are compatible with changes to the choice of toric model, such as the \emph{broken line convexity} of Cheung--Magee--N\'ajera Ch\'avez \cite{cmnc}. However, in general this requires equipping $\trop U$ with a choice of scattering diagram so that one is able to talk about broken lines in $\trop U$. 

\paragraph{The cone $\sigma$ associated to $Y$.}
We now associate a cone $\sigma\subset \trop U$ to our partial compactification $(Y,D)$. The image of the three components of $D$ under the projection $\pi\colon Y\to \Aa^3_{x,y,t}$ are 
\[ \pi(D_1) = \VV(x,y), \quad \pi(D_2) = \VV(y) \quad \text{and} \quad \pi(D_3) = \VV(t). \]
Thus these three divisors correspond to the integral points $d_1=(1,1,0)$, $d_2=(0,1,0)$ and $d_3=(0,0,1)$ in $\trop U(\ZZ)$ respectively. We want to consider the convex hull of the three rays $\rho_i=\RR_{\geq0}d_i$ for $i=1,2,3$, but taken with respect to the integral affine structure on $U^{\text{trop}}$. 

\begin{prop} \label{prop!sigma}
Consider the cone $\sigma\subset \trop U\cong N_\RR$ which is defined in terms of the integral affine structure on $N_\RR$ as the closed convex cone cut out by the hyperplanes $(1,0,0)^\perp$, $(0,0,1)^\perp$, $(1,-2,0)^\perp$ and $(1,-1,-n)^\perp$. If the singular line of $\trop U$ is taken to be $\ell=\RR(a,nb,b)$ for some choice of $a,b\in \ZZ$ with $0<a<2b$, then $\sigma$ is the convex hull of the three rays $\rho_1$, $\rho_2$, $\rho_3$ when considered in the integral affine structure of $\trop U$. The cone $\sigma\subset U^{\text{trop}}$ has three faces, since the two faces of $\sigma$ (as a subset of $N_\RR$) cut out by $(1,-2,0)^\perp$ and $(1,-1,-n)^\perp$ form a single face which bends as it passes through $H^+$.
\end{prop}

\begin{proof}
In a general integral affine manifold with singularities there can be more that one straight line segment between two given points, and in our case there are at most two; a straight line in $\trop U\setminus {\overline H}^+$, or a line that bends by $M_\ell^{\pm1}$ as it passes through $H^+$. (Note that once a line passes through $H$, either through $H^-$ or $H^+$, it does not pass through it again.) The assumption on the location of the singular line $\ell=\RR(a,nb,b)$ means is that $\ell$ intersects the interior of the cone $\sigma$, and avoids being contained in one of the faces of $\sigma$ (which happens if $a=0$ or $a=2n$).

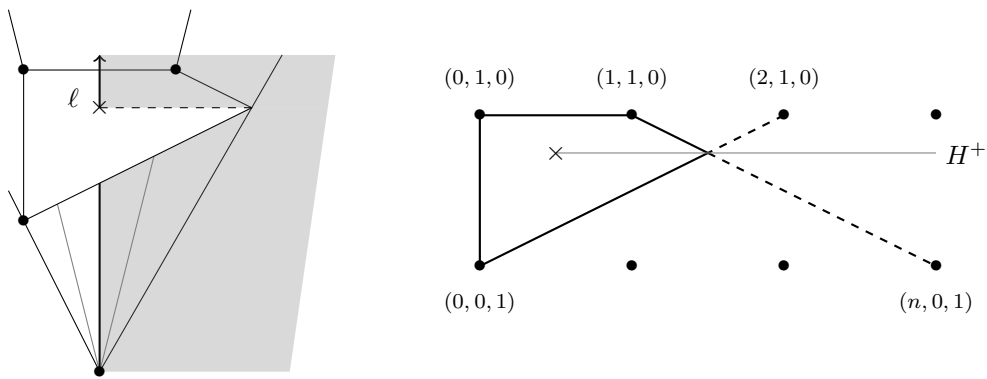
\begin{figure}[htbp]
\begin{center}
\begin{tikzpicture}[scale=2,font=\small]
    \begin{scope}[xshift=3cm, yshift = -0.3cm]  
   \draw[thick, dashed] (2,1) -- (6/4,3/4) (6/4,3/4) -- (3,0);
   \draw[thick] (6/4,3/4) -- (0,0) -- (0,1) -- (1,1) -- (6/4,3/4);
   
   \draw[gray] (1/2,3/4) -- (3,3/4);
   
   \node at (0,0) [label={below:\scriptsize $(0,0,1)$}] {$\bullet$};
   \node at (1,0) {$\bullet$};
   \node at (2,0) {$\bullet$};
   \node at (3,0) [label={below:\scriptsize $(n,0,1)$}] {$\bullet$};
   \node at (0,1) [label={above:\scriptsize $(0,1,0)$}]{$\bullet$};
   \node at (1,1) [label={above:\scriptsize $(1,1,0)$}]{$\bullet$};
   \node at (2,1) [label={above:\scriptsize $(2,1,0)$}]{$\bullet$};
   \node at (3,1) {$\bullet$};
   \node at (1/2,3/4) {$\times$};
   \node at (3.2,3/4) {$H^+$};
   
    \end{scope}
   
    \draw (1/2,-1) -- (1/2-1.2*1/2,-1+1.2);   
    \draw[gray] (1/2,-1) -- (0,1);
    \draw[gray] (1/2,-1) -- (1,1);
    \draw (0,1) -- (1/2-1.2*1/2,-1+2.4);
    \draw (1,1) -- (1/2+1.2*1/2,-1+2.4);
    \draw (1/2,-1) -- (1/2+1.2,-1+1.2*7/4);
    \draw[thick] (1/2,-1) -- (1/2,3/4);
    \fill[gray,opacity=0.3] (1/2,-1) -- (1/2,3/4) -- (3/4+5/4,3/4) -- (1/2+5/4,-1) -- cycle;
    \draw[fill=white] (0,0) -- (0,1) -- (1,1) -- (6/4, 3/4) -- (0,0); 
    \draw[thick,->] (1/2,3/4) -- (1/2,-1+1.2*7/4);
    \fill[gray,opacity=0.3] (1/2,3/4) -- (1/2,-1+1.2*7/4) -- (1/2+1.2*1/4+5/4,-1+1.2*7/4) -- (3/4+5/4,3/4) -- cycle;
    \node[above] at (1/2,3/4) [label={left:$\ell$}] {};
    \node at (0,0) {$\bullet$};
    \node at (0,1) {$\bullet$};
    \node at (1,1) {$\bullet$};
    \node at (1/2,-1) {$\bullet$};
    \node at (1/2, 3/4) {$\times$};
    \draw[dashed] (1/2, 3/4) -- (6/4,3/4);
\end{tikzpicture}
\caption{The cone $\sigma\subset \trop U$ and the slice of $\sigma$ by the hyperplane $\langle {\cdot}, (0,1,1)\rangle =1$.}
\label{fig!convex-cone}
\end{center}
\end{figure}

Now, consider any three points $v_1\in\rho_1$, $v_2\in \rho_2$ and $v_3\in \rho_3$. Since $v_1$ and $v_2$ belong to the same side of $H$ there is only one straight line between them which must lie in the hyperplane $\langle v_1,v_2\rangle = (0,0,1)^\perp$. Since $v_2$ and $v_3$ lie on opposite sides of $H$ we have at most two straight lines between them (cf.\ Figure~\ref{fig!bending}). We certainly have the usual straight line from $v_2$ to $v_3$ which is contained in the hyperplane $\langle v_2,v_3\rangle = (1,0,0)^\perp$ and must pass through $H^-$ (and hence won't bend) by the decision to choose $a>0$. We possibly have a second line that starts off at $v_2$ heading towards $M_\ell^{-1}v_3$ and then bends back towards $v_3$ if it hits $H^+$ (which happens only if $a<nb$). The first segment of this second line is contained in the hyperplane $\langle v_2,M_\ell^{-1}v_3 \rangle = (n,0,-1)^\perp$ and the second segment is contained in $\langle M_\ell v_2,v_3 \rangle = (1,-1,-n)^\perp$. Similarly there are at most two straight lines from $v_1$ to $v_3$. We possibly have the usual straight line from $v_1$ to $v_3$ contained in the hyperplane $\langle v_1,v_3\rangle = (1,-1,0)^\perp$, since if $a>nb$ then this line will hit the halfplane $H^-$ and carry on without bending. However we always have line that starts of at $v_1$ travelling towards $M_\ell^{-1}v_3$, hits $H^+$ (by virtue of our decision to choose $a<2nb$) and bends back towards $v_3$. The two segments of this line are contained in the hyperplanes $\langle v_1,M_\ell^{-1}v_3 \rangle = (1,-1,-n)^\perp$ and $\langle M_\ell v_1,v_3 \rangle = (1,-2,0)^\perp$. Viewed in the integral affine structure of $\trop U$, these two hyperplanes cut out a single face of $\sigma$. Altogether we find that the cone $\sigma$ contains any line between any two points of $\rho_1,\rho_2,\rho_3$ and thus must be the convex hull in the integral affine structure of $\trop U$.
\end{proof}

\subsection{The developing space of $\sigma$}

An alternative way to view the cone $\sigma$ is inside the \emph{developing space} of $\trop U$. This is the integral affine manifold $\mathcal D$ obtained as the universal covering space $\delta\colon \mathcal D\to \trop U\setminus \ell$. Since $\trop U\setminus \ell$ is homeomorphic to $\RR\times(\RR^2\setminus 0)$, the developing space $\mathcal D$ is an infinite sheeted cover of $\trop U\setminus \ell$. This is depicted in Figure~\ref{fig!developing}, where travelling anticlockwise around $\ell$ we wind from the blue sheet onto the black sheet, and then finally onto the red sheet. 
\begin{figure}[htbp]
\begin{center}
\begin{tikzpicture}[scale=2]
   \draw (13/4,3/4) -- (13/4,-1/4) -- (-3/2,-1/4) -- (-3/2,3/2) -- (7/2,3/2) -- (7/2,3/4);
   \draw[blue] (7/2,3/4) -- (7/2,-1/2) -- (-5/4,-1/2) -- (-5/4,5/4) -- (3/4,3/4);
   \draw[red] (13/4,3/4) -- (13/4,7/4) -- (11/4,7/4) -- (3/4,3/4);

   \draw[thick,red] (6/4,3/4) -- (2,1) -- (5/4,1);
   \draw[thick] (0,0) -- (6/4,3/4) -- (1,1) -- (0,1) -- (0,0);
   \draw[thick,blue] (6/4,3/4) -- (3,0) -- (-1,1) -- (-1/4,1);
   
   \draw[dashed] (3/4,3/4) -- (4,3/4);
   
   \node at (-1,0) {$\bullet$};
   \node at (0,0) [label={[fill=white]below:\scriptsize $(0,0,1)$}] {$\bullet$};
   \node at (1,0) {$\bullet$};
   \node at (2,0) {$\bullet$};
   \node at (3,0) [label={[fill=white]below:\scriptsize $(n,0,1)$}] {$\bullet$};
   \node at (-1,1) [label={[fill=white]above:\scriptsize $(-1,1,0)$}]{$\bullet$};
   \node at (0,1) [label={[fill=white]above:\scriptsize $(0,1,0)$}]{$\bullet$};
   \node at (1,1) [label={[fill=white]above:\scriptsize $(1,1,0)$}]{$\bullet$};
   \node at (2,1) [label={[fill=white]above:\scriptsize $(2,1,0)$}]{$\bullet$};
   \node at (3,1) {$\bullet$};
   \node at (3/4,3/4) {$\times$};
   
   \fill[black, opacity=0.2] (0,0) -- (6/4,3/4) -- (1,1) -- (0,1) -- (0,0);
   \fill[blue, opacity=0.2] (6/4,3/4) -- (3,0) -- (-1,1) -- (-1/4,1) -- (3/4,3/4) -- cycle;
   \fill[red, opacity=0.2] (6/4,3/4) -- (2,1) -- (5/4,1) -- (3/4,3/4) -- cycle;
\end{tikzpicture}
\caption{The developing space $\mathcal D$ for the cone $\sigma\subset U^{\text{trop}}$.}
\label{fig!developing}
\end{center}
\end{figure}
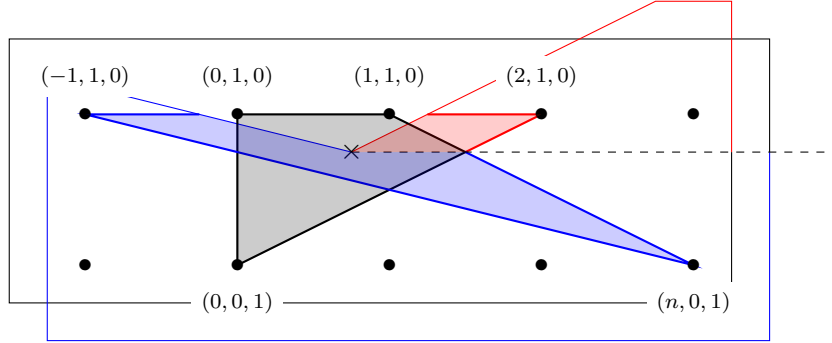
Every time we move onto a new sheet the integral points of $\trop U$ are multiplied by the monodromy matrix $M_\ell$. To carry out calculations using the integral affine structure of $\trop U$ we must stick to a consistent sheet of $\delta$. For example, another way to show that $\sigma$ is a convex cone with respect to the affine structure on $\trop U$ would be to show that $\delta^{-1}(\sigma\setminus \ell)$ is convex in every sheet of the developing map.


\subsection{A combinatorial model of the $(-2,0)$-flop}

\paragraph{Two possible subdivisions of $\sigma$.}
As we saw in the proof of Proposition~\ref{prop!sigma}, by moving the line of singularities $\ell=\RR(a,nb,b)$ (which corresponds to changing the choice of toric model for $U$ by Remark~\ref{rmk!moving-ell}) the cone $\sigma$ contains a second (interior) line between any two points of $\rho_1$ and $\rho_3$ if $a>nb$. Thus in this case $\sigma$ contains a plane which is linear with respect to the integral affine structure on $\trop U$ and which we can use to subdivide $\sigma$. Similarly, if $a<nb$ we can subdivide $\sigma$ by inserting a plane between $\rho_2$ and $\rho_3$. Thus in the limit $a=nb$ the cone $\sigma$ can be subdivided in exactly two different ways, as in Figure~\ref{fig!bending}. This is the motivation for the choice $\ell=\RR(n,n,1)$ made in Remark~\ref{rmk!moving-ell}.

\begin{figure}[htbp]
\begin{center}
\begin{tikzpicture}[scale = 1.5]
    \begin{scope}
    \node at (0,0) [label={below:$\rho_3$}] {$\bullet$};
    \node at (0,1) [label={above:$\rho_2$}] {$\bullet$};
    \node at (1,1) [label={above:$\rho_1$}] {$\bullet$};
    \draw (0,0) -- (0,1) -- (1,1) -- (6/4, 3/4) -- (0,0);  
    \node at (3/4-0.3, 3/4) {$\times$};  
    \draw[dashed] (3/4-0.3, 3/4) -- (2,3/4);
    \draw (0,0) -- (3/4,3/4) -- (0,1);
    \node at (1/2,-0.7) {$a<nb$};
    \end{scope}
    \begin{scope}[xshift = 3cm]
    \node at (0,0) [label={below:$\rho_3$}] {$\bullet$};
    \node at (0,1) [label={above:$\rho_2$}] {$\bullet$};
    \node at (1,1) [label={above:$\rho_1$}] {$\bullet$};
    \draw (0,0) -- (0,1) -- (1,1) -- (6/4, 3/4) -- (0,0);  
    \node at (3/4, 3/4) {$\times$};  
    \draw[dashed] (3/4, 3/4) -- (2,3/4);
    \draw (0,0) -- (1,1) (3/4,3/4) -- (0,1);
    \node at (1/2,-0.7) {$a=nb$};
    \end{scope}
    \begin{scope}[xshift = 6cm]
    \node at (0,0) [label={below:$\rho_3$}] {$\bullet$};
    \node at (0,1) [label={above:$\rho_2$}] {$\bullet$};
    \node at (1,1) [label={above:$\rho_1$}] {$\bullet$};
    \draw (0,0) -- (0,1) -- (1,1) -- (6/4, 3/4) -- (0,0);  
    \node at (3/4+0.3, 3/4) {$\times$};  
    \draw[dashed] (3/4+0.3, 3/4) -- (2,3/4);
    \draw (0,0) -- (1,1);
    \node at (1/2,-0.7) {$a>nb$};
    \end{scope}
\end{tikzpicture}
\caption{The two different subdivisions of $\sigma$.}
\label{fig!bending}
\end{center}
\end{figure}
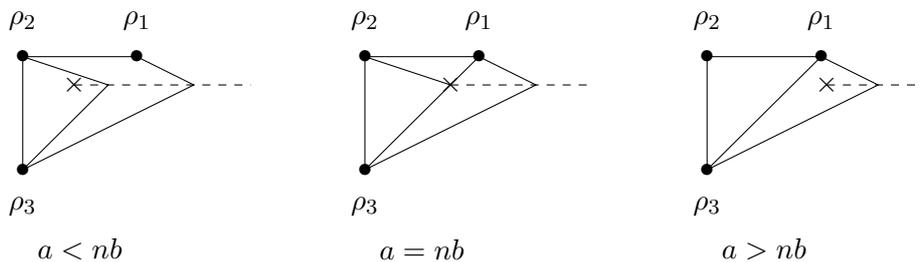

We take the two subdivisions of Figure~\ref{fig!bending} to be the combinatorial model of our flop. The analogue of Figure~\ref{fig!atiyah-flop} becomes Figure~\ref{fig!20flop}.

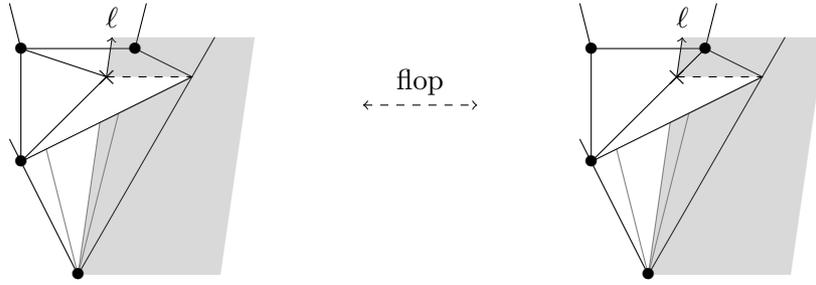
\begin{figure}[htbp]
\begin{center}
\begin{tikzpicture}[scale = 1.5]
 
    \draw (1/2,-1) -- (1/2-1.2*1/2,-1+1.2);   
    \draw[gray] (1/2,-1) -- (0,1);
    \draw[gray] (1/2,-1) -- (1,1);
    \draw (0,1) -- (1/2-1.2*1/2,-1+2.4);
    \draw (1,1) -- (1/2+1.2*1/2,-1+2.4);
    \draw (1/2,-1) -- (1/2+1.2,-1+1.2*7/4);
    \draw[gray] (1/2,-1) -- (3/4,3/4);
    \fill[gray,opacity=0.3] (1/2,-1) -- (3/4,3/4) -- (3/4+5/4,3/4) -- (1/2+5/4,-1) -- cycle;
    \draw[fill=white] (0,0) -- (0,1) -- (1,1) -- (6/4, 3/4) -- (0,0); 
    \draw[->] (3/4,3/4) -- (1/2+1.2*1/4,-1+1.2*7/4);
    \fill[gray,opacity=0.3] (3/4,3/4) -- (1/2+1.2*1/4,-1+1.2*7/4) -- (1/2+1.2*1/4+5/4,-1+1.2*7/4) -- (3/4+5/4,3/4) -- cycle;
    \node[above] at (1/2+1.2*1/4,-1+1.2*7/4) {$\ell$};
    \node at (0,0) {$\bullet$};
    \node at (0,1) {$\bullet$};
    \node at (1,1) {$\bullet$};
    \node at (1/2,-1) {$\bullet$};
    \node at (3/4, 3/4) {$\times$};
    \draw[dashed] (3/4, 3/4) -- (6/4,3/4);
    \draw (0,0) -- (3/4,3/4) -- (0,1);
    
    \draw[<->,dashed] (3,0.5) to node[midway,above] {flop} (4,0.5);
    
    \begin{scope}[xshift = 5cm]    
    \draw (1/2,-1) -- (1/2-1.2*1/2,-1+1.2);   
    \draw[gray] (1/2,-1) -- (0,1);
    \draw[gray] (1/2,-1) -- (1,1);
    \draw (0,1) -- (1/2-1.2*1/2,-1+2.4);
    \draw (1,1) -- (1/2+1.2*1/2,-1+2.4);
    \draw (1/2,-1) -- (1/2+1.2,-1+1.2*7/4);
    \draw[gray] (1/2,-1) -- (3/4,3/4);
    \fill[gray,opacity=0.3] (1/2,-1) -- (3/4,3/4) -- (3/4+5/4,3/4) -- (1/2+5/4,-1) -- cycle;
    \draw[fill=white] (0,0) -- (0,1) -- (1,1) -- (6/4, 3/4) -- (0,0); 
    \draw[->] (3/4,3/4) -- (1/2+1.2*1/4,-1+1.2*7/4);
    \fill[gray,opacity=0.3] (3/4,3/4) -- (1/2+1.2*1/4,-1+1.2*7/4) -- (1/2+1.2*1/4+5/4,-1+1.2*7/4) -- (3/4+5/4,3/4) -- cycle;
    \node[above] at (1/2+1.2*1/4,-1+1.2*7/4) {$\ell$};
    \node at (0,0) {$\bullet$};
    \node at (0,1) {$\bullet$};
    \node at (1,1) {$\bullet$};
    \node at (1/2,-1) {$\bullet$};
    \node at (3/4, 3/4) {$\times$};
    \draw[dashed] (3/4, 3/4) -- (6/4,3/4);
    \draw (0,0) -- (1,1);
    \end{scope}
    
\end{tikzpicture}
\caption{The combinatorial picture of a $(-2,0)$-flop.}
\label{fig!20flop}
\end{center}
\end{figure}

\begin{thm}\label{thm!20-flop}
Let $f\colon (X,\widetilde D)\to(Y,D)$ be the flopping extraction
\[ f\colon X = \VV\left( y\alpha - z\beta, \: x\alpha - (y+t^n)\beta \right)\subset \Aa^4\times \PP^1_{\alpha:\beta} \to Y, \] 
over $P\in Y$, where $\widetilde D=\sum_{i=1}^3\widetilde D_i\subset X$ is the strict transform of $D\subset Y$. Then the intersection complex of $(X,\widetilde D)$ is dual to the lefthand subdivision of $\sigma$ appearing in Figure~\ref{fig!20flop}, where the three rays of $\sigma$ correspond to the three boundary divisors $\widetilde D_1$, $\widetilde D_2$, $\widetilde D_3$ and the subdividing plane corresponds to the exceptional curve $C=f^{-1}(P)$. Moreover the righthand subdivision of $\sigma$ is dual to the intersection complex of the other flopping extraction from $P\in Y$.
\end{thm}

\begin{proof}
As described in the proof of Lemma~\ref{lem!resolution}, the small resolution is covered by two affine charts $X = \Aa^3_{\beta,z,t}\cup \Aa^3_{\alpha,x,t}$ and the three boundary divisors are given by $\widetilde D_1 =\VV(\beta)$, $\widetilde D_2 =\VV(z)$, $\widetilde D_3 =\VV(t)$ in the first chart and $\widetilde D_1 =\emptyset$, $\widetilde D_2 =\VV(\alpha x - t^n)$, $\widetilde D_3 =\VV(t)$ in the second. The exceptional curve is $C=\widetilde D_2\cap \widetilde D_3$ and the three boundary divisors intersect transversely everywhere, except for the point $Q:=(0,0,0)\in \Aa^3_{\alpha,x,t}$ where $\widetilde D_2$ and $\widetilde D_3$ look like an $A_{n-1}$ singularity and a plane meeting along their toric boundary strata. Thus the intersection complex of $\widetilde D$ is given by the rightmost diagram in Figure~\ref{fig!intersection-cx}. The verification for the other side of the flop is given by a similar calculation.
\end{proof}

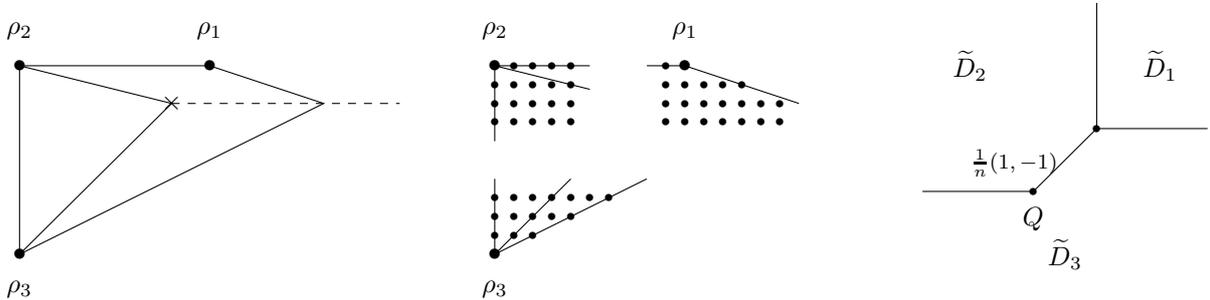
\begin{figure}[htbp]
\begin{center}
\begin{center}\begin{tikzpicture}[scale=2.5,font=\small]

  \begin{scope}[xshift = 0cm]
    \node at (0.0,1.0) {\tiny $\bullet$};
    \node at (0.1,1.0) {\tiny $\bullet$};
    \node at (0.2,1.0) {\tiny $\bullet$};
    \node at (0.3,1.0) {\tiny $\bullet$};
    \node at (0.4,1.0) {\tiny $\bullet$};
    \node at (0.0,0.9) {\tiny $\bullet$};
    \node at (0.1,0.9) {\tiny $\bullet$};
    \node at (0.2,0.9) {\tiny $\bullet$};
    \node at (0.3,0.9) {\tiny $\bullet$};
    \node at (0.4,0.9) {\tiny $\bullet$};
    \node at (0.0,0.8) {\tiny $\bullet$};
    \node at (0.1,0.8) {\tiny $\bullet$};
    \node at (0.2,0.8) {\tiny $\bullet$};
    \node at (0.3,0.8) {\tiny $\bullet$};
    \node at (0.4,0.8) {\tiny $\bullet$};
    \node at (0.0,0.7) {\tiny $\bullet$};
    \node at (0.1,0.7) {\tiny $\bullet$};
    \node at (0.2,0.7) {\tiny $\bullet$};
    \node at (0.3,0.7) {\tiny $\bullet$};
    \node at (0.4,0.7) {\tiny $\bullet$};

    \node at (1.0,1) {\tiny $\bullet$};
    \node at (0.9,1) {\tiny $\bullet$};
    \node at (0.9,0.9) {\tiny $\bullet$};
    \node at (1.0,0.9) {\tiny $\bullet$};
    \node at (1.1,0.9) {\tiny $\bullet$};
    \node at (1.2,0.9) {\tiny $\bullet$};
    \node at (1.3,0.9) {\tiny $\bullet$};
    \node at (0.9,0.8) {\tiny $\bullet$};
    \node at (1.0,0.8) {\tiny $\bullet$};
    \node at (1.1,0.8) {\tiny $\bullet$};
    \node at (1.2,0.8) {\tiny $\bullet$};
    \node at (1.3,0.8) {\tiny $\bullet$};
    \node at (1.4,0.8) {\tiny $\bullet$};
    \node at (0.9,0.7) {\tiny $\bullet$};
    \node at (1.0,0.7) {\tiny $\bullet$};
    \node at (1.1,0.7) {\tiny $\bullet$};
    \node at (1.2,0.7) {\tiny $\bullet$};
    \node at (1.3,0.7) {\tiny $\bullet$};
    \node at (1.4,0.7) {\tiny $\bullet$};
    \node at (1.5,0.8) {\tiny $\bullet$};
    \node at (1.5,0.7) {\tiny $\bullet$};

    \node at (0.0,0.0) {\tiny $\bullet$};
    \node at (0.0,0.1) {\tiny $\bullet$};
    \node at (0.0,0.2) {\tiny $\bullet$};
    \node at (0.1,0.1) {\tiny $\bullet$};
    \node at (0.1,0.2) {\tiny $\bullet$};
    \node at (0.2,0.1) {\tiny $\bullet$};
    \node at (0.2,0.2) {\tiny $\bullet$};
    \node at (0.3,0.2) {\tiny $\bullet$};
    \node at (0.4,0.2) {\tiny $\bullet$};
    \node at (0.0,0.3) {\tiny $\bullet$};
    \node at (0.1,0.3) {\tiny $\bullet$};
    \node at (0.2,0.3) {\tiny $\bullet$};
    \node at (0.3,0.3) {\tiny $\bullet$};
    \node at (0.4,0.3) {\tiny $\bullet$};
    \node at (0.5,0.3) {\tiny $\bullet$};
    \node at (0.6,0.3) {\tiny $\bullet$};
    
    \node at (0,0) [label={below:$\rho_3$}] {$\bullet$};
    \node at (0,1) [label={above:$\rho_2$}] {$\bullet$};
    \node at (1,1) [label={above:$\rho_1$}] {$\bullet$};
    \draw (0,0.4) -- (0,0) -- (0.4,0.4) (0,0) -- (0.8,0.4);
    \draw (0,0.6) -- (0,1) -- (0.5,0.875) (0,1) -- (0.5,1);
    \draw (0.8,1) -- (1,1) -- (1.6,0.8); 
  \end{scope}
  \begin{scope}[xshift = -2.5cm]
    \node at (0,0) [label={below:$\rho_3$}] {$\bullet$};
    \node at (0,1) [label={above:$\rho_2$}] {$\bullet$};
    \node at (1,1) [label={above:$\rho_1$}] {$\bullet$};
    \node at (4/5,4/5) {$\times$};
    \draw (0,0) -- (0,1) -- (1,1) -- (8/5, 4/5) -- (0,0); 
    \draw[dashed] (4/5, 4/5) -- (2,4/5);
    \draw (0,0) -- (4/5,4/5) -- (0,1);
  \end{scope}
    
  \begin{scope}[xshift = 2.5cm]
    \draw (-1/4,1/3) -- (1/3,1/3) -- (2/3,2/3) -- (2/3,4/3)  (2/3,2/3) -- (5/4,2/3);
    \node at (1/3-0.1,1/3+0.15) {\scriptsize $\tfrac1n(1,-1)$};
    \node at (1/3,1/3) [label={[label distance=-3pt]below:$Q$}]{\tiny $\bullet$};
    \node at (2/3,2/3) {\tiny $\bullet$};
    
    \node at (0,1) {$\widetilde D_2$};
    \node at (1,1) {$\widetilde D_1$};
    \node at (0.5,0) {$\widetilde D_3$};
  \end{scope}    
\end{tikzpicture}\end{center}
\caption{The intersection complex of the boundary divisor on one side of the flop.}
\label{fig!intersection-cx}
\end{center}
\end{figure}

Moreover, the manner of the intersection between boundary components can also be read off from subdivision of $\sigma$ by looking at the local fan structure around each of the rays $\rho_1,\rho_2,\rho_3$. For example, restricting to one sheet of the developing map, in a neighbourhood of the ray $\rho_2=\RR_{\geq0}(0,1,0)$ we see a two dimensional fan with rays $(0,-1)$, $(n,-1)$ and $(1,0)$ (as seen in the top left corner of the middle diagram of Figure~\ref{fig!intersection-cx}) corresponding to the three curves of $\widetilde D_2\cap (\widetilde D_1\cup\widetilde D_3)$. The fact that $\widetilde D_2$ has an $A_{n-1}$ singularity at $Q\in X$ is seen from the fact that $\langle (0,1), (n,1) \rangle$ is the cone of an $A_{n-1}$-singularity.

\section{Reid's pagoda} \label{sect!pagoda}

\subsection{The pagoda} \label{sect!pagoda-description}
\emph{Reid's pagoda} is a geometric construction introduced by Reid \cite[\S5]{reid} to describe the geometry of a $(-2,0)$-flop. In particular, Reid shows that if $(C\subset X)$ is a $(-2,0)$-flopping curve of width $n\geq2$ then blowup of $C$ gives a morphism $\pi_1\colon (F_1\subset X_1)\to (C\subset X)$ with exceptional divisor a Hirzebruch surface $F_1\cong \mathbb F_2$. Moreover, if $C_1$ is the negative curve of $F_1$ then $C_1 \subset X_1$ is a $(-2,0)$-flopping curve of width $n-1$ \cite[Theorem 5.4]{reid}. By induction we have a sequence of blowups
\[ (F_n\subset X_n) \stackrel{\pi_n}{\longrightarrow} (C_{n-1}\subset F_{n-1}\subset X_{n-1}) \longrightarrow \cdots \longrightarrow (C_1\subset F_1\subset X_1) \stackrel{\pi_1}{\longrightarrow} (C\subset X) \]
in which each curve $C_k\subset X_k$ is a $(-2,0)$-curve, except for the last $C_{n-1}\subset X_{n-1}$ which is a $(-1,-1)$-curve, and each exceptional divisor is $F_k\cong \mathbb F_2$, except for the last which is $F_n\cong\PP^1\times\PP^1$. 

As an abuse of notation let $F_k\subset X_n$ denote the strict transform of $F_k$ in $X_n$ for $k=1,\ldots,n$ and let $C_k\subset X_n$ denote the negative curve of $\mathbb F_2\cong F_k\subset X_n$ for $k=1,\ldots,n-1$. The \emph{pagoda} is the union of divisors $\bigcup_{k=1}^nF_k$ in $X_n$. These divisors form a tower in which each divisor $F_k$ intersects the divisor $F_{k+1}$ above it along the curve $C_k=F_k\cap F_{k+1}$, in such a way that $(C_k|_{F_k})^2=-2$ and $(C_k|_{F_{k+1}})^2=2$ (cf.\ Figure~\ref{fig!pagoda}). In particular for the last divisor $C_{n-1}\subset F_n$ is a curve of bidegree $(1,1)$.

Now there are two possible contractions $X_n\to X$ or $X_n\to X'$ which contract the pagoda back down to a line, by starting with one of the two possible contractions for the divisor $F_n\subset X_n$ (being, as it is, the exceptional divisor over a $(-1,-1)$-flopping curve). These two possible contractions give a resolution $X\leftarrow X_n\rightarrow X'$ of the flop $\phi_n\colon X\dashrightarrow X'$ in terms of smooth blowups and blowdowns.

\subsection{Obtaining Reid's pagoda by subdividing $\sigma$}

We will now see how to obtain Reid's pagoda by subdividing $\sigma\subset\trop U$ into a union of smooth cones, in exactly the same manner that one would resolve a toric singularity by toric blowups. 

\paragraph{A height function on $\sigma$.}
We consider the height function $h\colon \trop U\to \RR$ given by $h(v)=b+c$ for $v=(a,b,c)$. Note that since the monodromy around $\ell$ acts by $M_\ell v = (a+b-nc,b,c)$ we have $h(v) = h(M_\ell v)$ and thus this is a well-defined linear function on $\trop U$. Moreover $h$ takes integral values on $\trop U(\ZZ)$, $h(v)\geq0$ for all $v\in\sigma$ and the slices $\sigma^{h=\lambda} = \{ v\in \sigma : h(v) = \lambda \}$ for $\lambda\geq0$ are all compact.  

\paragraph{The subdivision.}
We will proceed by induction on $n\geq1$, so we let $Y_n=Y$ and $\sigma_n=\sigma$. The height 1 slice $\sigma_n^{h=1}$ contains three integral points. These are the points $d_1=(1,1,0)$, $d_2=(0,1,0)$ and $d_3=(0,0,1)$ generating the rays of $\sigma_n$, and they correspond to the three boundary divisors $D_1,D_2,D_3\subset Y_n$. Looking at the height 2 slice $\sigma_n^{h=2}$ we see seven points of $\trop U(\ZZ)$, as shown in the righthand side of Figure~\ref{fig!slice}. In particular there is an interior integral point $f_1=(1,1,1)\in\sigma_n^{h=2}\cap \trop U(\ZZ)$. 
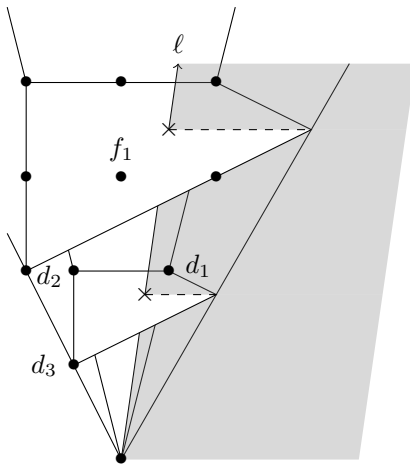
\begin{figure}[htbp]
\begin{center}
\begin{tikzpicture}[scale = 2.5,font=\small]
    
    \draw (1/2,-1) -- (1/2-1.2*1/2,-1+1.2);   
    \draw (1/2,-1) -- (1/2-1.2*1/2,-1+2.4);
    \draw (1/2,-1) -- (1/2+1.2*1/2,-1+2.4);
    \draw (1/2,-1) -- (1/2+1.2,-1+1.2*7/4);
    \draw (1/2,-1) -- (5/8,-1/8);
    \fill[gray,opacity=0.3] (1/2,-1) -- (5/8,-1/8) -- (5/8+5/4,-1/8) -- (1/2+5/4,-1) -- cycle;
    \draw[fill=white] (1/4,-1/2) -- (1/4,0) -- (3/4,0) -- (1,-1/8) -- cycle;
    \draw (3/4,3/4) -- (5/8,-1/8);
    \fill[gray,opacity=0.3] (3/4,3/4) -- (5/8,-1/8) -- (5/8+5/4,-1/8) -- (3/4+5/4,3/4) -- cycle;
    \draw[fill=white] (0,0) -- (0,1) -- (1,1) -- (6/4, 3/4) -- (0,0); 
    \draw[->] (3/4,3/4) -- (1/2+1.2*1/4,-1+1.2*7/4);
    \fill[gray,opacity=0.3] (3/4,3/4) -- (1/2+1.2*1/4,-1+1.2*7/4) -- (1/2+1.2*1/4+5/4,-1+1.2*7/4) -- (3/4+5/4,3/4) -- cycle;

    \node[above] at (1/2+1.2*1/4,-1+1.2*7/4) {$\ell$};
    
    \node at (0,0) {$\bullet$};
    \node at (0,1/2) {$\bullet$};
    \node at (0,1) {$\bullet$};
    \node at (1/2,1/2) [label={[label distance=-4pt]90: $f_1$}]  {$\bullet$};
    \node at (1/2,1) {$\bullet$};
    \node at (1,1/2) {$\bullet$};
    \node at (1,1) {$\bullet$};
    
    \node at (1/4,-1/2) [label={[label distance=-4pt]180: $d_3$}] {$\bullet$};
    \node at (1/4,0) [label={[yshift=-2pt,label distance=-6pt]180: $d_2$}] {$\bullet$};
    \node at (3/4,0) [label={[yshift=2pt,label distance=-4pt]0: $d_1$}] {$\bullet$};
    
    \node at (1/2,-1) {$\bullet$};
    \node at (3/4, 3/4) {$\times$};
    \node at (5/8, -1/8) {$\times$};  
    \draw[dashed] (3/4, 3/4) -- (6/4,3/4) (5/8,-1/8) -- (1,-1/8);
    
\end{tikzpicture}
\caption{The integral point $f_1\in \trop U(\ZZ)$ in the height 2 slice $\sigma^{h=2}$.}
\label{fig!slice}
\end{center}
\end{figure}

\begin{lem}\label{lem!pagoda}
The refinement of the cone $\sigma_n$ along the ray $\RR_{\geq0}f_1$ subdivides $\sigma_n$ into two smooth cones and a cone isomorphic to $\sigma_{n-1}$.
\end{lem}

\begin{proof}
The refinement along $\RR_{\geq0}f_1$ divides $\sigma_n$ into three cones $\sigma_n=\tau_1\cup\tau_2\cup\tau_3$ where $\tau_1=\langle f_1,d_2,d_3 \rangle$, $\tau_2=\langle d_1,f_1,d_3 \rangle$ and $\tau_3=\langle d_1,d_2,f_1 \rangle$. Note that, due to the convenient choice of location for $\ell\subset \trop U$, all three of $\tau_1$, $\tau_2$ and $\tau_3$ are convex cones with respect to the integral affine structure on $\trop U$. Indeed, $\tau_1$ is convex since any line from $d_2$ to a point on the face $\langle f_1,d_3\rangle$ must be a straight line in $N_\RR$, since a line that bends along $H^+$ would have to pass through $\ell$. Similarly $\tau_2$ and $\tau_3$ are also convex cones.
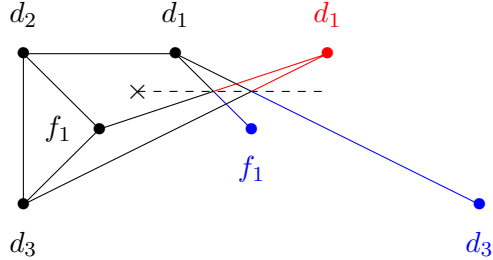
\begin{center}\begin{tikzpicture}[scale = 2]
    \node at (0,0) [label={below:$d_3$}] {$\bullet$};
    \node at (0,1) [label={above:$d_2$}] {$\bullet$};
    \node at (1/2,1/2) [label={left:$f_1$}]{$\bullet$};
    \node at (1,1) [label={above:$d_1$}] {$\bullet$};
    \node[red] at (2,1) [label={above:\textcolor{red}{$d_1$}}] {$\bullet$};
    \node[blue] at (3/2, 1/2) [label={below:\textcolor{blue}{$f_1$}}] {$\bullet$};
    \node[blue] at (3, 0) [label={below:\textcolor{blue}{$d_3$}}] {$\bullet$};
    \draw (0,0) -- (0,1) -- (1,1) -- (6/4, 3/4) -- (0,0);  
    \node at (3/4, 3/4) {$\times$};  
    \draw[dashed] (3/4, 3/4) -- (2,3/4);
    \draw (0,0) -- (1/2,1/2) -- (0,1) (1/2,1/2) -- (1/2 + 3/4, 1/2 + 1/4) -- (1,1);
    \draw[red] (1/2 + 3/4, 1/2 + 1/4) -- (2,1) -- (6/4,3/4);
    \draw[blue] (6/4,3/4) -- (3,0) (1/2 + 3/4, 1/2 + 1/4) -- (3/2, 1/2);
\end{tikzpicture}\end{center}
By looking at the appropriate sheet of the developing map for $\sigma_n$ we see that
\[ \tau_1 \cong \langle (0,1,0), (0,0,1), (1,1,1) \rangle, \qquad \tau_2 \cong \langle (2,1,0), (0,0,1), (1,1,1) \rangle \]
and these are easily both seen to be smooth. To see that the last cone $\tau_3$ is isomorphic to $\sigma_{n-1}$ consider the matrix
\[ P = \begin{pmatrix}
1 & 0 & -1 \\
0 & 1 & -1 \\
0 & 0 & 1 
\end{pmatrix} \in \SL(3,\ZZ). \]
This maps $d_1,d_2,f_1$ to the points $(1,1,0),(0,1,0),(0,0,1)$ and the singular ray $\ell$ to $\RR_{\geq0}(n-1,n-1,1)$. The monodromy around $\ell$ in this new configuration is $PM_nP^{-1}=M_{n-1}$. 

Finally, if $n=1$ then $\ell$ coincides with the ray $\RR_{\geq0}f_1$ and this refinement is a subdivision of $\sigma_1$ into three smooth cones.
\end{proof}

\paragraph{The pagoda.}
We now define the sequence of integral points
\[ \{ f_i = (i,i,1) \in\sigma_n^{h=i+1} \cap \trop U(\ZZ) : i = 1,\ldots,n \}, \]
where we note that the last point $f_n$ is a primitive generator for the singular ray $\ell$. By subdividing $\sigma_n$ along the ray $\RR_{\geq0}f_i$ for $i=1,\ldots,n$ inductively we arrive at Figure~\ref{fig!pagoda}, where the subdivision of $\sigma_n$, displayed on the lefthand side is dual to the resolution of $(P\in Y)$ displayed on the righthand side
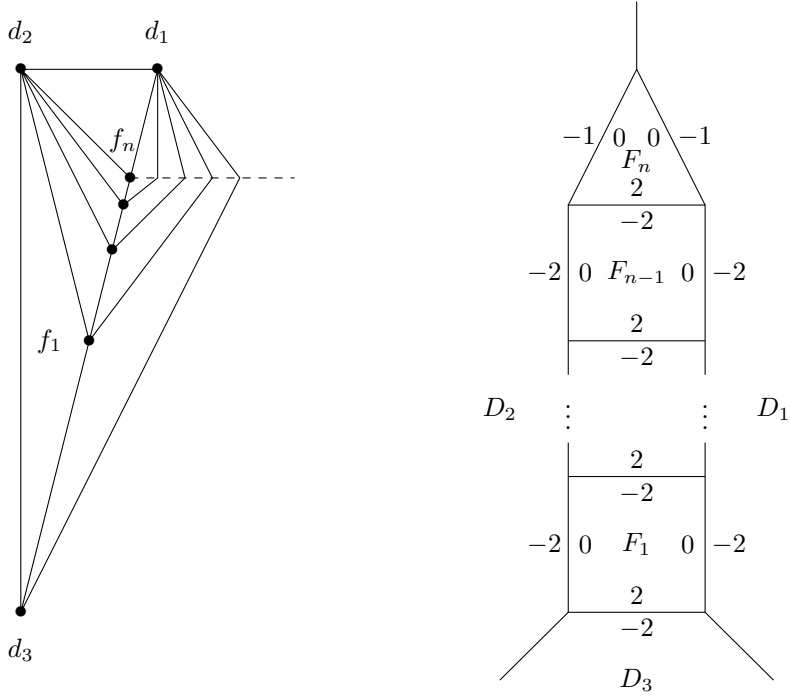
\begin{figure}[htbp]
\begin{center}
\begin{center}\begin{tikzpicture}[scale=0.9,font=\small]

  \begin{scope}[xshift = -8cm, yscale = 8, xscale = 2]
    \node at (0,0) [label={below:$d_3$}] {$\bullet$};
    \node at (0,1) [label={above:$d_2$}] {$\bullet$};
    \node at (1/2,1/2) [label={left:$f_1$}] {$\bullet$};
    \node at (2/3,2/3) {$\bullet$};
    \node at (3/4,3/4) {$\bullet$};
    \node at (1,1) [label={above:$d_1$}] {$\bullet$};
    \draw (0,0) -- (0,1) -- (1,1) -- (8/5, 4/5) -- (0,0);  
    \node at (4/5, 4/5) [label={[xshift=-0.1cm]above:$f_n$}] {$\bullet$};  
    \draw[dashed] (4/5, 4/5) -- (2,4/5);
    \draw (0,0) -- (1,1) (1/2,1/2) -- (0,1) -- (2/3,2/3) (3/4,3/4) -- (0,1) -- (4/5,4/5);
    \draw (1/2,1/2) -- (7/5,4/5) -- (1,1);
    \draw (2/3,2/3) -- (6/5,4/5) -- (1,1);
    \draw (3/4,3/4) -- (5/5,4/5) -- (1,1);
    
  \end{scope}
  
  \draw (0,2.5) -- (0,0) -- (2,0) -- (2,2.5);
  \draw (2,3.5) -- (2,6) -- (0,6) -- (0,3.5);
  \draw (0,2) -- (2,2) (0,4) -- (2,4) (0,6) -- (1,8) -- (2,6);
  \draw (1,8) -- (1,9) (0,0) -- (-1,-1) (2,0) -- (3,-1);
  \node at (1,1) {$F_1$};
  \node at (1,5) {$F_{n-1}$};
  \node at (1,6.65) {$F_{n}$};
  \node at (-1,3) {$D_2$};
  \node at (3,3) {$D_1$};
  \node at (1,-1) {$D_3$};
  \node at (1,0.25) {$2$};
  \node at (1,1.75) {$-2$};
  \node at (1,2.25) {$2$};
  \node at (1,3.75) {$-2$};
  \node at (1,4.25) {$2$};
  \node at (1,5.75) {$-2$};
  \node at (1,6.25) {$2$};
  \node at (0.25,1) {$0$};
  \node at (0,3) {$\vdots$};
  \node at (0.25,5) {$0$};
  \node at (0.75,7) {$0$};
  \node at (1.75,1) {$0$};
  \node at (2,3) {$\vdots$};
  \node at (1.75,5) {$0$};
  \node at (1.25,7) {$0$};

  \node at (1,-0.25) {$-2$};
  \node at (2.35,1) {$-2$};
  \node at (2.35,5) {$-2$};
  \node at (1.85,7) {$-1$};
  \node at (-0.35,1) {$-2$};
  \node at (-0.35,5) {$-2$};
  \node at (0.15,7) {$-1$};
  
\end{tikzpicture}\end{center}
\caption{The subdivision of $\sigma$ corresponding to Reid's pagoda.}
\label{fig!pagoda}
\end{center}
\end{figure}

\begin{thm}\label{thm!pagoda}
The resolution of $\sigma_n$ into smooth cones obtained by this sequence of subdivisions is dual to the intersection complex of the boundary in Reid's pagoda, where the integral point $f_i\in\trop U(\ZZ)$ corresponds to the $i$th Hirzebruch surface $F_i\cong \mathbb F_2$ for $i=1,\ldots,n-1$ and the last point $f_n$ corresponds to the final exceptional divisor $F_n\cong \PP^1\times\PP^1$. By restricting to consistent sheets of the developing map $\delta$, the self-intersection numbers of the curves in the pagoda can be calculated in the same way that they are in toric geometry. 
\end{thm}

\begin{proof}
One can check that after the inductive sequence of blowups used to construct the pagoda in \S\ref{sect!pagoda-description}, the strict transform of the boundary divisors $D_1,D_2,D_3$ intersect the exceptional divisors $F_1,\ldots,F_n$ as shown in the righthand side of Figure~\ref{fig!pagoda}. Thus this subdivision of $\sigma_n$ is dual to the intersection complex of the pagoda. 

To see that the intersection numbers of the curves in the pagoda can now been seen to agree with the standard toric computation. For example to show that the normal bundle of $C_{n-1} = F_n\cap F_{n-1}$ is $\mathcal O(2,-2)$ consider the equation
\[ d_2 + M_\ell d_1 = (0,1,0) + (2,1,0) = (2,2,0) = 2(n,n,1) - 2(n-1,n-1,1) = 2f_n - 2f_{n-1}, \]
where we use $M_\ell d_1$, rather than $d_1$, to stay on a consistent sheet of $\delta$. If the pagoda was a toric variety this computation would tell us that $(C_{n-1}|_{F_{n-1}})^2 = -2$ and $(C_{n-1}|_{F_n})^2 = 2$, as expected.
\end{proof}

\begin{rmk}
Note that the final exceptional divisor corresponding to $f_n$ is the only component of the pagoda which doesn't look like a toric surface, since it appears with a non-toric boundary divisor. This accounts for the singularity in the integral affine structure around the ray $\ell\subset \sigma$. 
\end{rmk}

\section{Constructing the dual cone} \label{sect!dual-side}


We now explain how to construct a cone $\sigma^\star\subset \trop V$ in an integral affine manifold with singularities which is dual to $\sigma\subset \trop U$.  As described in \S\ref{sect!generalising}, this could be done by tropicalising the Fock--Goncharov dual cluster variety $V$ with a similar calculation to that of Proposition~\ref{prop!Utrop}. However we don't do this and instead choose to give a construction of $\trop V$ which is more intrinsically related to $U$, although we retain the notation $\trop V$.

In fact any integral affine manifold with singularities $B$ (with the additional data of a decomposition of $B$ into polyhedral cones and a polarisation) can be dualised by a general construction called the \emph{discrete Legendre transform} (see e.g.\ \cite{rudd}). The dual integral affine manifold with singularities $B^\star$ is homeomorphic to $B$ and has the same singular locus $\Delta\subset B^\star$. The smooth locus $B\setminus \Delta$ has a local system $\Lambda$ of tangent vectors, and a dual local system $\Lambda^\star$ of cotangent vectors. On $B^\star\setminus \Delta$ the local systems $\Lambda$ and $\Lambda^\star$ are exchanged. The monodromy $M_\gamma$ generated by a small loop $\gamma$ around a codimension 2 component of $\Delta$ is replaced by the inverse transpose $M_\gamma^\star=(M_\gamma^{-1})^t$.

\subsection{The dual integral affine manifold $\trop V$}

Recall that the integral points $\trop V(\ZZ) \subset \trop V$ are supposed to correspond to an additive basis of the coordinate ring $\CC[U]$. Since $U$ has coordinate ring
\[ \CC[U] = \CC\left[x,y^{\pm1},z,t^{\pm1}\right] / \left(xz - y(y+t^n)\right) \]
the following two sets of monomials 
\[ S_1 := \{ x^ay^bt^c : (a,b,c)\in \ZZ_{\geq0}\times\ZZ^2 \} \quad \text{and} \quad S_2 :=  \{ z^ay^bt^c : (a,b,c)\in \ZZ_{\geq0}\times\ZZ^2 \} \] 
form an additive basis for $\CC[U]$, after we account for the fact that we have double-counted all monomials of the form $S_1\cap S_2 = \{ y^bt^c: (b,c)\in \ZZ^2 \}$. 

\begin{defn}
We will call the set $S_1\cup S_2$ the \emph{theta basis} of $\CC[U]$, and refer to the monomials that appear in this set as \emph{theta functions} on $U$.
\end{defn}

Since the monomials of the set $S_1$ are all linearly independent in $\CC[U]$, any divisorial valuation $\nu\colon \CC(U)^\times\to\ZZ$, corresponding to a point $\nu\in \trop U(\ZZ)$, can be evaluated as a linear function $\nu(x^ay^bt^c) = a\nu(x) + b\nu(y) + c\nu(t)$ on $S_1$. Thus we consider $S_1=V_1(\ZZ)$ as the integral points of the integral affine manifold with boundary $V_1:=\RR_{\geq0}\times\RR^2$. Similarly we consider $S_2 = V_2(\ZZ)$ as the integral points of $V_2:=\RR_{\geq0}\times\RR^2$. We can now construct $\trop V$ by gluing these two halves together along the plane $\Pi=\{0\}\times \RR^2$, in such a way that the underlying real manifold of $\trop V$ is homeomorphic to $\RR^3$ but the integral affine structure on $\trop V$ is dual to that of $\trop U$.

Note that the obvious identification $\trop V := V_1 \sqcup V_2 / \left( (0,b,c) \sim (0,b,c)\right)$ produces a real manifold homeomorphic to $\RR^3$ and therefore lets us identify the integral points $\trop V(\ZZ)$ with the lattice $\ZZ^3\subset \RR^3$. However it does not specify an integral affine structure on $\trop V$, since there is no information as to how continue drawing straight lines that pass from one chart into the other. To fix this, for every $p\in\Pi$ which is not in the singular locus of $\trop V$ we need to provide $A_p\in\SL(3,\ZZ)$ such that for any piecewise linear map $\gamma\colon[0,1]\to \trop V$ and $t\in(0,1)$ with $\gamma(t) = p\in \Pi$, $\gamma([t-\varepsilon,t])\subset V_1$ and $\gamma([t,t+\varepsilon])\subset V_2$ for some $\varepsilon>0$, the matrix $A_p$ gives us the tangent vector $\gamma'(t+\epsilon) = A_p\gamma'(t-\epsilon)$ of a straight line $\gamma$ extending into $V_2$.

If $U$ had been a toric hypersurface defined by the binomial equation $xz = y^2$ then we recover $M_\RR$ by choosing $A_p$ to be the matrix $A_+ = \left(\begin{smallmatrix} 1&0&0\\-2&1&0\\0&0&1 \end{smallmatrix}\right)$ for all $p\in \Pi$, since this identifies $z,y,t$ with $x^{-1}y^2,y,t$. Alternatively, if $U$ had been defined by the binomial equation $xz = yt^n$ then we recover $M_\RR$ by choosing $A_p$ to be the matrix $A_- = \left(\begin{smallmatrix} 1&0&0\\-1&1&0\\-n&0&1 \end{smallmatrix}\right)$ for all $p\in \Pi$. Since $U$ is defined by the trinomial equation $xz=y^2+yt^n$ we use both by picking a line $\ell^\star\subset \Pi$ which will be the singular locus of $\trop V$ and which splits $\Pi\setminus \ell^\star = \Pi^+\sqcup\Pi^-$. Now we use $A_p=A_\pm$ for all $p\in \Pi^\pm$, cf.\ Figure~\ref{fig!dual-cone}.

\paragraph{The dual cone $\sigma^\star$.}
To construct the dual cone $\sigma^\star$ we note that the monomials from $S_1$ which are regular on $Y$ are precisely those of the form $x^ay^bt^c$ for $a,b,c\in\ZZ_{\geq0}^3$, and therefore these generate a cone $\RR_{\geq0}^3\subset V_1$. Similarly the monomials from $S_2$ which are regular on $Y$ are those of the form $z^ay^bt^c$ for $a,b,c\in\ZZ_{\geq0}^3$, and generate a cone $\RR_{\geq0}^3\subset V_2$. We define $\sigma^\star$ to be the closed convex cone in $\trop V$ obtained by gluing these two halves together according to the identification described above. This is shown in Figure~\ref{fig!dual-cone}.

\begin{figure}[htbp]
\begin{center}
\begin{tikzpicture}[scale=1.5,font=\small]
   \node at (0,0) [label={below:$(1,0,0)$}] {$\bullet$};
   \node at (1,0) [label={below:$(0,1,0)$}] {$\bullet$};
   \node at (2,0) [label={below:$(-1,2,0)$}] {$\bullet$};
   \node at (1,1) [label={left:$(0,0,1)$}] {$\bullet$};
   \node at (1,1/2) {$\times$};
   \draw (0,0) -- (2,0) -- (1,1) -- cycle;
   \draw[gray] (1,1/2) -- (1,3/2);
   
   \begin{scope}[xshift=6cm]
   \node at (0,0) [label={below:$(1,0,0)$}] {$\bullet$};
   \node at (1,0) [label={-45:$(0,1,0)$}] {$\bullet$};
   \node at (4/3,1) [label={right:$(-1,1,n)$}] {$\bullet$};
   \node at (1,1) [label={left:$(0,0,1)$}] {$\bullet$};
   \node at (1,1/2) {$\times$};
   \draw (0,0) -- (1,0) -- (4/3,1) -- (1,1) -- cycle;
   \draw[gray] (1,1/2) -- (1,-1/2);
   \end{scope}
   
   \begin{scope}[xshift=3.25cm, yshift=1cm]
   \node at (0-0.2,0) [label={below:$x$}] {$\bullet$};
   \node at (1-0.2,0) [label={below:$y$}] {$\bullet$};
   \node at (1-0.2,1) [label={above:$t$}] {$\bullet$};
   \node at (1+0.2,0) [label={below:$y$}] {$\bullet$};
   \node at (1+0.2,1) [label={above:$t$}] {$\bullet$};
   \node at (2+0.2,1) [label={above:$z$}] {$\bullet$};
   \draw (0.8,1) -- (-0.2,0) -- (0.8,0);
   \draw (1.2,1) -- (2.2,1) -- (1.2,0);
   \draw[dashed] (0.8,1) -- (0.8,0) (1.2,1) -- (1.2,0);
   \draw[gray] (1,-0.5) -- (1,1.5);
   \node[fill=white] at (1,0.5) {$\times$};
   \node[fill=white] at (1.1,-0.6) {$\Pi^+$};
   \node[fill=white] at (1.1, 1.6) {$\Pi^-$};
   
   \draw[->] (-0.5,0) -- (-1.25,-0.5);
   \draw[->] (2,0) -- (2.75,-0.5);
   
   \node at (0.5-0.2,1.3) {$V_1$};
   \node at (1.5+0.2,1.3) {$V_2$};
   \end{scope}
\end{tikzpicture}
\caption{The construction of the dual cone $\sigma^\star$.}
\label{fig!dual-cone}
\end{center}
\end{figure}
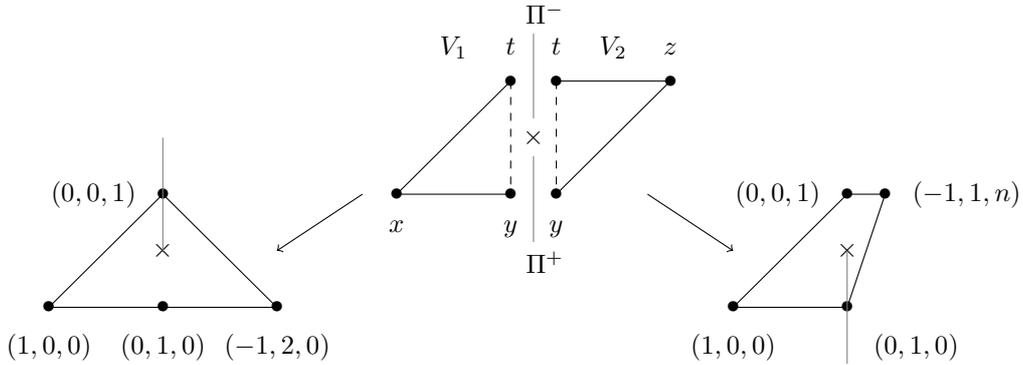

\paragraph{Monodromy.}
Because of the difference in the way we identify $V_1$ and $V_2$ either side of the line $\ell^\star$ there is a monodromy action as we travel anticlockwise around $\ell^\star\subset \trop V$. This monodromy is represented by the matrix
\[ M_{\ell^\star} = A_-^{-1}A_+  = \begin{pmatrix}
1 & 0 & 0 \\
-1 & 1 & 0 \\
 n & 0 & 1
\end{pmatrix} \]
and we note that this is the inverse transpose of the monodromy $M_\ell$ around $\ell\subset \trop U$, as expected.

\paragraph{Dual pairing.}
Since points $\nu\in\trop U(\ZZ)$ correspond to divisorial valuations $\nu\colon \CC(U)^\times \to \ZZ$ and points $\vartheta\in\trop V(\ZZ)$ correspond to monomials $\vartheta\in\CC[U]$ we have an intersection pairing given by evaluating $\nu$ on $\vartheta$
\[ \langle{\cdot},{\cdot}\rangle \colon \trop U(\ZZ)\times \trop V(\ZZ)\to \ZZ, \qquad \langle \nu,\vartheta\rangle = \nu(\vartheta) \]
and this extends to a pairing $\langle{\cdot},{\cdot}\rangle \colon \trop U\times \trop V\to \RR$ by scaling (to extend from $\ZZ$ to $\QQ$) and continuity (to extend from $\QQ$ to $\RR$). By construction this pairing is cooked up to be linear either side of the hyperplane $\Pi\subset \trop U$ and similarly it is also linear either side of the hyperplane $H\subset \trop U$, as shown in Figure~\ref{fig!pairing}.

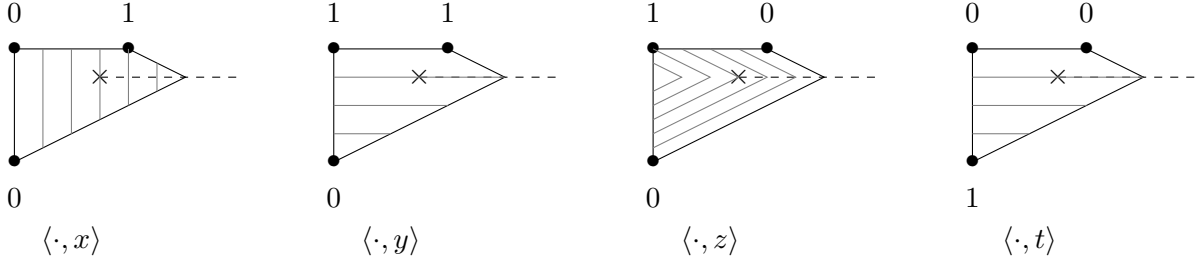
\begin{figure}[htbp]
\begin{center}
\begin{tikzpicture}[scale = 1.5]
    \begin{scope}
    \node at (1/2,-0.7) {$\langle{\cdot},x\rangle$};
    \node at (0,0) [label={below:$0$}] {$\bullet$};
    \node at (0,1) [label={above:$0$}] {$\bullet$};
    \node at (1,1) [label={above:$1$}] {$\bullet$};
    \draw (0,0) -- (0,1) -- (1,1) -- (6/4, 3/4) -- (0,0);  
    \node at (3/4, 3/4) {$\times$};  
    \draw[dashed] (3/4, 3/4) -- (2,3/4);
    
    \draw[gray] (1/4,1) -- (1/4,1/8);
    \draw[gray] (2/4,1) -- (2/4,2/8);
    \draw[gray] (3/4,1) -- (3/4,3/8);
    \draw[gray] (4/4,1) -- (4/4,4/8);
    \draw[gray] (5/4,7/8) -- (5/4,5/8);
    \end{scope}
    \begin{scope}[xshift = 2.8cm]
    \node at (1/2,-0.7) {$\langle{\cdot},y\rangle$};
    \node at (0,0) [label={below:$0$}] {$\bullet$};
    \node at (0,1) [label={above:$1$}] {$\bullet$};
    \node at (1,1) [label={above:$1$}] {$\bullet$};
    \draw (0,0) -- (0,1) -- (1,1) -- (6/4, 3/4) -- (0,0);  
    \node at (3/4, 3/4) {$\times$};  
    \draw[dashed] (3/4, 3/4) -- (2,3/4);
    
    \draw[gray] (0,3/4) -- (6/4,3/4);
    \draw[gray] (0,2/4) -- (4/4,2/4);
    \draw[gray] (0,1/4) -- (2/4,1/4);
    \end{scope}
    \begin{scope}[xshift = 5.6cm]
    \node at (1/2,-0.7) {$\langle{\cdot},z\rangle$};
    \node at (0,0) [label={below:$0$}] {$\bullet$};
    \node at (0,1) [label={above:$1$}] {$\bullet$};
    \node at (1,1) [label={above:$0$}] {$\bullet$};
    \draw (0,0) -- (0,1) -- (1,1) -- (6/4, 3/4) -- (0,0);  
    \node at (3/4, 3/4) {$\times$};  
    \draw[dashed] (3/4, 3/4) -- (2,3/4);
    
    \draw[gray] (0,7/8) --  (1/4,3/4) -- (0,5/8);
    \draw[gray] (0,1) --  (2/4,3/4) -- (0,4/8);
    \draw[gray] (1/4,1) --  (3/4,3/4) -- (0,3/8);
    \draw[gray] (2/4,1) --  (4/4,3/4) -- (0,2/8);
    \draw[gray] (3/4,1) --  (5/4,3/4) -- (0,1/8);
    \end{scope}
    \begin{scope}[xshift = 8.4cm]
    \node at (1/2,-0.7) {$\langle{\cdot},t\rangle$};
    \node at (0,0) [label={below:$1$}] {$\bullet$};
    \node at (0,1) [label={above:$0$}] {$\bullet$};
    \node at (1,1) [label={above:$0$}] {$\bullet$};
    \draw (0,0) -- (0,1) -- (1,1) -- (6/4, 3/4) -- (0,0);  
    \node at (3/4, 3/4) {$\times$};  
    \draw[dashed] (3/4, 3/4) -- (2,3/4);
    
    \draw[gray] (0,3/4) -- (6/4,3/4);
    \draw[gray] (0,2/4) -- (4/4,2/4);
    \draw[gray] (0,1/4) -- (2/4,1/4);
    \end{scope}
\end{tikzpicture}
\caption{The functions $\langle {\cdot}, x\rangle$ etc., viewed on the height 1 slice of $\sigma$.}
\label{fig!pairing}
\end{center}
\end{figure}

\subsection{Recovering the coordinate patches on a $(-2,0)$-flop}

Now that we have constructed $\trop U$ and $\trop V$ with their dual pairing, we can dualise any convex cone $\tau\subset \trop U$ as we would in the toric setting; by defining $\tau^\star = \{ v\in \trop V : \langle u,v \rangle \geq 0 \text{ for all } u \in \tau \}$. In particular, we would now like to construct the two small resolutions $X,X'$ appearing on either side of $(-2,0)$-flop in an analogous way to which the Atiyah flop can be constructed in toric geometry. Namely, we would like to take one of the subdivisions $\sigma = \tau_1\cup \tau_2$ appearing as one side of our combinatorial model of the $(-2,0)$-flop (as in Figure~\ref{fig!20flop}), dualise each half $\tau_i$ to get $\tau_i^\star\subset \trop V$, build an affine variety $W_i=\Spec \CC[\tau_i^\star\cap \trop V(\ZZ)]$ and finally glue $W_1$ and $W_2$ together to recover either $X$ or $X'$.  

Unfortunately we run into at least two problems as we try to do this. Let $\sigma=\tau_1\cup\tau_2$ be the subdivision on the lefthand side of Figure~\ref{fig!20flop}, where $\tau_1=\langle \rho_2,\rho_3,\ell \rangle$ is the left half and $\tau_2=\overline{\sigma\setminus \tau_1}$ is the right half. The first problem is that only $\tau_1$ is actually convex in $\trop U$. Indeed $\tau_2$ contains all three rays of $\sigma$, and so the convex hull of $\tau_2$ should be the whole of $\sigma$. Nevertheless we can dualise $\tau_1$.

\begin{prop}
If we dualise $\tau_1\subset \trop U$ then the integral points of $\tau_1^\star$ correspond to theta functions generating the coordinate ring of the affine variety
\[ W_1 = \Spec\CC[\tau_1^\star\cap \trop V(\ZZ)] = \CC\left[x,y,y^{-1}z,t\right] / \big(xz - y(y+t^n)\big) \]
which comprises one affine patch of one side of the flop over $(P\in Y)$.
\end{prop}

\begin{proof}
Let $f\colon X\to Y$ be the small resolution given by blowing up the Weil divisor $D_2=\VV(y,z)\subset Y$. Then $X$ can be glued together from two affine patches $X= W_1\cup W_2$ where $\CC[W_1]=\CC[Y][y^{-1}z]$ and $\CC[W_2]=\CC[Y][yz^{-1}]$. Now we can construct $\tau_1^\star$ by dualising $\tau_1$ in each half $(\tau_1^\star)_1\subset V_1$ and $(\tau_2^\star)_1\subset V_2$ and identifying the two halves together along $\Pi$ to get $\tau_1^\star\subset \trop V$, exactly as we did to construct $\sigma^\star$ above. If $v=(n,n,1)$ is the primitive generator of $\ell$, then the value of the pairing between $d_1,d_2,v$ and $x,y,z,t$ is given in the following table. 
\[ \begin{array}{|c|cccc|}\hline
 & x & y & z & t \\\hline
d_2 & 0 & 1 & 1 & 0 \\
d_3 & 0 & 0 & 0 & 1 \\ 
v & n & n & n & 1 \\ \hline
\end{array} \]
Therefore the only monomials in $S_1$ that are positive on $\tau_1$ are those corresponding to integral points in the cone generated $x,y,t$ and the only monomials in $S_2$ that are positive on $\tau_1$ are those corresponding to integral points in the cone generated $y,y^{-1}z,t$. Thus we obtain $W_1= \Spec\CC[\tau_1^\star\cap \trop V(\ZZ)]$. 
\end{proof}

The second problem is that the second affine patch of $X$ has coordinate ring $W_2=\CC[Y][yz^{-1}]$, but the coordinate $yz^{-1} = x(y+t^n)^{-1}$, is not a regular function on $U$ and hence cannot be written in terms of the theta functions on $U$. 



\section{Mirror symmetry} \label{sect!mirror}

As an interesting consequence of the construction of the cone $\sigma\subset\trop U$ we can produce a singularity $(P^\star\in Y^\star)$ which is the compactification of the mirror cluster variety $V$ corresponding to the dual cone $\sigma^\star\subset \trop V$. In other words, we construct an affine 3-fold singularity $Y^\star = \Spec R$, whose coordinate ring $R = \CC[\vartheta_v : v \in \sigma\cap \trop U(\ZZ)]$ is generated by the theta functions on $V$ corresponding to the integral points of $\sigma$. 

\subsection{The theta functions on $V$}
We first describe the theta functions on $V$ in terms of a cluster torus chart $\TT^\star=\CC^\times\otimes M \hookrightarrow V$. Since $U$ is obtained by blowing up the curve $Z=\VV\left(1+u^{(0,-1,n)}\right)$ in the boundary component $B_{(1,0,0)}$ of a toric pair $(T,B)$, by the discussion in \S\ref{sect!generalising} the Fock--Goncharov dual $V$ is obtained by blowing up the curve $Z^\star=\VV\left(1+v^{(1,0,0)}\right)$ in the boundary component $B^\star_{(0,-1,n)}$ of a toric pair $(T^\star,B^\star)$. From this description the theta functions on $V$ are given in terms of a \emph{scattering diagram} on $\trop U$. Let $\TT^\star = (\CC^\times)^3_{u,v,w}$ have coordinates $u = v^{(1,0,0)}$, $v = v^{(0,1,0)}$ and $w = v^{(0,0,1)}$. Since $V$ is given by a single nontoric blowup, the scattering diagram used to construct $\CC[V]$ is very simple and consists of a single wall $H=(0,-1,n)^\perp\subset \trop U$ with an attached scattering function $f_H=1 + u$. There is a wall crossing automorphism $\theta_H\colon \CC[\TT^\star] \to \CC[\TT^\star]$ given by
\[ \theta_H(u^av^bw^c) = u^av^bw^c(1+u)^{-\langle (a,b,c), (0,-1,n)\rangle}. \]
Pick a (generic, irrational) basepoint $p\in\trop U$ on one side of $H$, the side on which $\langle p,(0,1,-n)\rangle<0$ say. For any $q=(a,b,c)\in \trop U(\ZZ)$, an expansion of the theta function $\vartheta_q\in\CC[V]$ as a Laurent polynomial in $u,v,w$ can be described in terms of a special class of piecewise linear lines in $\trop U$ called \emph{broken lines}, which begin at $p$, bend finitely many times and end up parallel to $q$. The upshot in our simple situation is that if $p,q$ lie on the same side of $H$ then any such broken line is the usual straight line in $N_\RR$ that starts at $p$ and leaving parallel to $q$. This corresponds to the monomial $\vartheta_q=u^av^bw^c$. If $q$ lies on the other side of $H$ to $p$, then any broken line must cross $H$ and the formula for the theta function $\vartheta_q$ reduces to $\vartheta_q = \theta_H(u^av^bw^c)$. To summarise:
\begin{lem}
For a point $(a,b,c)\in\trop U(\ZZ)$, the theta function $\vartheta_{(a,b,c)}$ on $V$ is given by
\[ \vartheta_{(a,b,c)} = \begin{cases}
u^a v^b w^c & \text{if } b -nc \leq 0 \\
u^a v^b w^c (1+u)^{b-nc} & \text{if } b - nc \geq 0.
\end{cases} \] 
\end{lem}


\subsection{The mirror singularity}

Let $d_1,d_2,d_3,f_4,\ldots,f_{n+3}\in \sigma\cap \trop U(\ZZ)$ be the sequence of primitive integral points corresponding to the divisors $D_1,D_2,D_3,F_1,\ldots,F_n$ appearing in the construction of Reid's pagoda, and let $\vartheta_i=\vartheta_{d_i}$ for $i=1,2,3$ and $\vartheta_{i+3}=\vartheta_{f_i}$ for $i=1,\ldots,n$. 

\begin{thm}\label{thm!mirror-thm}
The mirror singularity to $(P\in Y)$ is given by the nonisolated affine 3-fold singularity $(P^\star\in Y^\star) = (0\in \VV(I)\subset \Aa^{n+3}_{\vartheta_1,\ldots,\vartheta_{n+3}})$, where $I$ is the ideal of codimension $n$ defined by the $2\times 2$ minors of the following $2\times n$ matrix
\[ \begin{pmatrix}
\vartheta_1+\vartheta_2 & \vartheta_3 & \vartheta_4 & \cdots & \vartheta_{n+1} & \vartheta_{n+2} \\
\vartheta_1\vartheta_2  & \vartheta_4 & \vartheta_5 & \cdots & \vartheta_{n+2} & \vartheta_{n+3} 
\end{pmatrix}. \]
\end{thm}

\begin{proof}
Let $R\subset \CC[Y^\star]$ be the subring generated by the claimed set of generators $\vartheta_1,\ldots,\vartheta_{n+3}$. Note that the equations generating $R$ are given by the minors of the matrix, where the ratio between the rows is given by
\[ uv = \frac{\vartheta_1\vartheta_2}{\vartheta_1 + \vartheta_2} = \frac{\vartheta_4}{\vartheta_3} = \cdots = \frac{\vartheta_{n+3}}{\vartheta_{n+2}}, \]
and these minors define a reduced, irreducible affine variety of dimension 3. We will now prove that $R= \CC[Y^\star]$ by induction on $n$, so we will let $Y_n=Y$ and $\sigma_n=\sigma$. 

\paragraph{Base case.} When $n=1$ subdivide $\sigma_1$ into the five smooth cones pictured in Figure~\ref{fig!subdivision}(a), corresponding to the six theta functions $\vartheta_1,\ldots,\vartheta_4, \vartheta_{(0,1,1)}$ and $\vartheta_{(2,1,1)}$. Each cone is smooth, and if $q=av_i+bv_j+cv_k$ belongs to the cone $\langle v_i,v_j,v_k\rangle$ with $a,b,c\in\ZZ_{\geq0}$ then we easily see that $\vartheta_q=\vartheta_{v_i}^a\vartheta_{v_j}^b\vartheta_{v_k}^c$, so these six theta functions generate $\CC[Y_1^\star]$. From the relations $\vartheta_{(0,1,1)} = \vartheta_2\vartheta_3 - \vartheta_4$ and $\vartheta_{(2,1,1)} = \vartheta_1\vartheta_3 - \vartheta_4$ it follows that $\vartheta_{(0,1,1)},\vartheta_{(2,1,1)}\in R$, and thus the claim holds in this case.

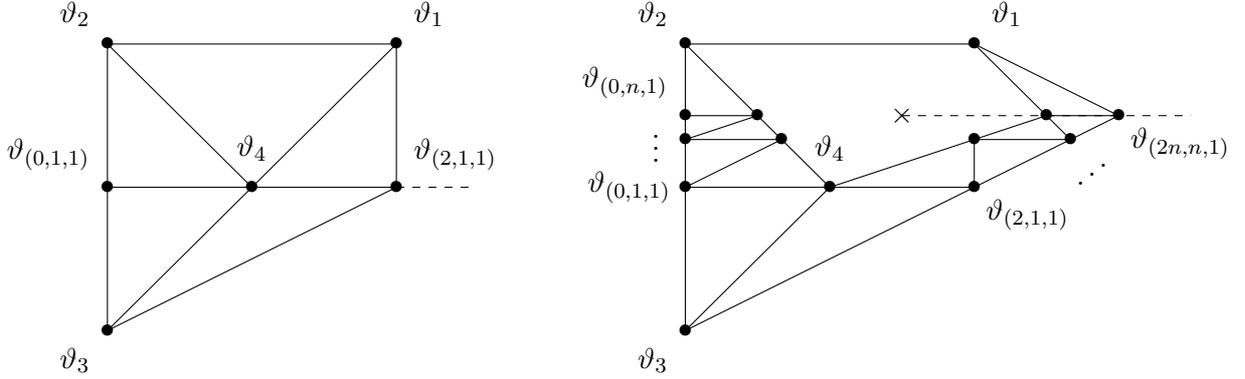
\begin{figure}[htbp]
\begin{center}
\begin{tikzpicture}[scale = 3.8]
    \begin{scope}
    \draw (0,0) -- (0,1) -- (1,1) -- (1, 1/2) -- (0,0);  
    \draw (0,1/2) -- (1,1/2) (0,0) -- (1,1) (1/2,1/2) -- (0,1);
    \draw[dashed] (1,1/2) -- (5/4,1/2);
    \node at (0,0) [label={[label distance=-5pt]225:$\vartheta_3$}] {$\bullet$};
    \node at (0,1/2) [label={[label distance=-5pt]135:$\vartheta_{(0,1,1)}$}] {$\bullet$};
    \node at (1/2,1/2) [label={above:$\vartheta_4$}] {$\bullet$};
    \node at (1,1/2) [label={[label distance=-5pt]45:$\vartheta_{(2,1,1)}$}] {$\bullet$};
    \node at (0,1) [label={[label distance=-5pt]135:$\vartheta_2$}] {$\bullet$};
    \node at (1,1) [label={[label distance=-5pt]45:$\vartheta_1$}] {$\bullet$};
    \end{scope}
    
    \begin{scope}[xshift = 2cm]
    \node (y5) at (3/4, 3/4) {$\times$};  
    \draw (0,0) -- (0,1) -- (1,1) -- (6/4, 3/4) -- (0,0);  
    \draw (0,1) -- (1/2,1/2) -- (5/4,3/4) -- (1,1) (0,0) -- (1/2,1/2);
    \draw[dashed]  (3/4,3/4) -- (7/4,3/4);
    \draw (0,3/4) -- (1/4,3/4) -- (0,2/3) -- (1/3,2/3) -- (0,1/2) -- (1,1/2) -- (1,2/3) -- (4/3,2/3) -- (5/4,3/4) -- (6/4,3/4);
    
    \node at (0,0) [label={[label distance=-5pt]225:$\vartheta_3$}] {$\bullet$};
    \node at (0,1) [label={[label distance=-5pt]135:$\vartheta_2$}] {$\bullet$};
    \node at (1,1) [label={[label distance=-5pt]45:$\vartheta_1$}] {$\bullet$};
    \node at (0,1/2) [label={[label distance=-5pt]180:$\vartheta_{(0,1,1)}$}] {$\bullet$};
    \node at (1/2,1/2) [label={90:$\vartheta_4$}] {$\bullet$};
    \node at (1,1/2) [label={[label distance=-8pt]-45:$\vartheta_{(2,1,1)}$}] {$\bullet$};
    \node at (0,2/3) [label={[label distance=-1pt]180:$\vdots$}] {$\bullet$};
    \node at (1/3,2/3) {$\bullet$};
    \node at (3/3,2/3) {$\bullet$};
    \node at (4/3,2/3) [label={[label distance=-12pt]-45:$\iddots$}] {$\bullet$};
    \node at (0/4,3/4) [label={[label distance=-5pt]135:$\vartheta_{(0,n,1)}$}] {$\bullet$};
    \node at (1/4,3/4) {$\bullet$};
    \node at (5/4,3/4) {$\bullet$};
    \node at (6/4,3/4) [label={[label distance=-8pt]-45:$\vartheta_{(2n,n,1)}$}] {$\bullet$};
    \end{scope}
\end{tikzpicture}
\caption{(a) The subdivision for the base case. (b) The subdivision for the induction step.}
\label{fig!subdivision}
\end{center}
\end{figure}

\paragraph{Induction step.}
For the induction step, write $\sigma_n = \tau_1\cup\tau_2\cup\tau_3$, as in the proof of Lemma~\ref{lem!pagoda}, by subdividing $\sigma_n$ at $f_1$. This is a subdivision of $\sigma_n$ into two smooth cones $\tau_1 = \left\langle d_2,d_3,f_1 \right\rangle$, $\tau_2 = \left\langle d_1,d_3,f_1 \right\rangle$ and $\tau_3\cong \sigma_{n-1}$. By the induction hypothesis all of the theta functions in $\tau_3$ can be written in terms of the theta functions $\vartheta_1,\vartheta_2,\vartheta_4,\ldots,\vartheta_{n+3}$. The cone $\tau_1$ can be further subdivided into $2n$ smooth cones
\[ \tau_2 = \bigcup \left\{ \renewcommand{\arraystretch}{1.5} \begin{matrix} 
\left\langle (0,k-1,1), (0,k,1), (1,k,1) \right\rangle & k=1,\ldots,n \\
\left\langle (0,k,1), (1,k,1), (1,k+1,1) \right\rangle & k=1,\ldots,n-1 \\
\left\langle (0,1,0), (0,n,1), (1,n,1) \right\rangle, 
\end{matrix} \right. \]
as shown in Figure~\ref{fig!subdivision}(b). To conclude that the theta functions in each one of these cones belongs to $R$, it is now enough to show that each $\vartheta_{(0,k,1)}\in R$ for $k=1,\ldots,n$. Since
\[ v^kw = \left(v + uv - uv\right)^kw = \sum_{i=0}^k (-1)^i\binom{k}{i}(v+uv)^{k-i} u^{i}v^{i}w \]
we must have $\vartheta_{(0,k,1)} = \sum_{i=0}^k (-1)^{i}\binom{k}{i}\vartheta_2^{k-i} \vartheta_{i+3}$. By a similar argument it follows that all of the theta functions in $\tau_2$ also belong to $R$, and thus $R= \CC[Y_n^\star]$.
\end{proof}

\subsection{A crepant resolution of $(P^\star\in Y^\star)$}
To describe the geometry of the situation we produce a crepant resolution of $(Y^\star,D^\star)$.

\begin{prop}
The pair $(Y^\star,D^\star)$ admits a crepant resolution $\mu\colon \widetilde Y^\star\to Y^\star$, with exceptional locus $\widetilde D_4^\star\cup \Gamma$, where $\widetilde D_4^\star$ is a divisor over the line of $cA_1$ singularities and $\Gamma\subset \widetilde D_3^\star$ is a rational curve with normal bundle $\mathcal N_{\Gamma/\widetilde Y^\star}=\sO(-1,-n)$ that meets $\widetilde D_4^\star$ transversely.
\end{prop}

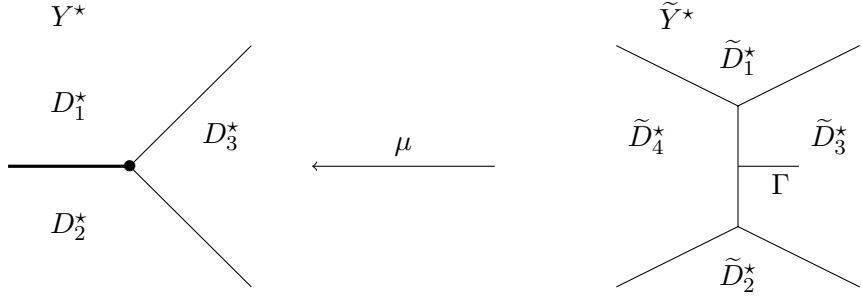
\begin{center}\begin{tikzpicture}[scale=0.8]
   
   \begin{scope}[xshift = 10cm]
   \draw (0,4) -- (2,3) -- (4,4) (2,3) -- (2,1) (0,0) -- (2,1) -- (4,0);
   \draw (2,2) -- (3,2);
   \node at (3.5,2.5) {$\widetilde D_3^\star$};
   \node at (0.5,2.5) {$\widetilde D_4^\star$};
   \node at (2,3.8) {$\widetilde D_1^\star$};
   \node at (2,0.2) {$\widetilde D_2^\star$};
   \node at (2.7,1.7) {$\Gamma$};
   
   \node at (1,4.5) {$\widetilde Y^\star$};
   \end{scope}
   
   \draw[->] (8,2) to node[midway,above] {$\mu$} (5,2);
   
%
   
   \begin{scope}[xshift = 0cm]
   \draw (4,4) -- (2,2) -- (4,0);
   \draw[very thick] (0,2) -- (2,2);
   \node at (2,2) {$\bullet$};
   \node at (3.5,2.5) {$D_3^\star$};
   \node at (1,3) {$D_1^\star$};
   \node at (1,1) {$D_2^\star$};
   
   \node at (1,4.5) {$Y^\star$};
   \end{scope}
\end{tikzpicture}\end{center}

\begin{proof}
We first consider the birational morphism $\phi \colon W\subset \Aa^{n+3}\times \PP^1_{\alpha_0,\alpha_1} \to Y^\star$, where $W$ is defined by the minors of the following matrix.
\[ \begin{pmatrix}
\alpha_0 & \vartheta_1+\vartheta_2 & \vartheta_3 & \vartheta_4 & \cdots & \vartheta_{n+1} & \vartheta_{n+2} \\
\alpha_1 & \vartheta_1\vartheta_2  & \vartheta_4 & \vartheta_5 & \cdots & \vartheta_{n+2} & \vartheta_{n+3} 
\end{pmatrix}. \]
This is a small morphism, with a single exceptional curve $\Gamma\cong \PP^1$ in the fibre over $P^\star$. We note that these equations imply that $\alpha_1^n\vartheta_3 = \alpha_0^n\vartheta_{n+3}$. Let $s = \alpha_0\alpha_1^{-1}$ and $t = \alpha_0^{-1}\alpha_1$ be the coordinates on the two affine patches of $\PP^1_{\alpha_0,\alpha_1}$. Then $W$ is covered by two affine charts $W=U_0\cup U_1$, where $U_i = \{ \alpha_i\neq0 \}$ for $i=0,1$. If $\alpha_0\neq0$ then we can eliminate the variables $\vartheta_4,\ldots,\vartheta_{n+3}$ to see that
\[ U_0 \cong \left( \vartheta_1\vartheta_2 = t(\vartheta_1+\vartheta_2) \right)\subset \Aa^4_{t,\vartheta_1,\vartheta_2,\vartheta_3}, \]
and, similarly, if $\alpha_1\neq0$ then we can eliminate the variables $\vartheta_3,\ldots,\vartheta_{n+2}$ to see that
\[ U_1 \cong \left( s\vartheta_1\vartheta_2 = \vartheta_1+\vartheta_2 \right)\subset \Aa^4_{s,\vartheta_1,\vartheta_2,\vartheta_{n+3}}. \]
These two charts are glued together by the map
\[ f\colon U_0\dashrightarrow U_1, \quad f(t,\vartheta_1,\vartheta_2,\vartheta_3) = \left(s^{-1},\vartheta_1,\vartheta_2,s^n\vartheta_{n+3}\right). \]
The patch $U_1$ is smooth and the patch $U_0$ has a line of $A_1$ singularities along $\ell=\Aa^1_{\vartheta_3}$. It is easy to verify that the crepant blowup of $\ell \subset U_1$ resolves the line of singularities to give our resolution $\mu \colon \widetilde Y^\star \to Y^\star$. Consider the affine patch in the blowup of $\ell\subset U_1$ given by 
\[ V = \left( \vartheta_1'\vartheta_2' = \vartheta_1' + \vartheta_2' \right)\subset \Aa^4_{t,\vartheta_1',\vartheta_2',\vartheta_{n+3}} \]
where $\vartheta_1 = \vartheta_1't$ and $\vartheta_2 = \vartheta_2't$. The two charts $V$ and $U_0$ cover the curve $\Gamma$ and the transition map is given by
\[ f\colon V\dashrightarrow U_1, \quad f(t,\vartheta_1',\vartheta_2',\vartheta_3) = \left(s^{-1},s\vartheta_1,s\vartheta_2,s^n\vartheta_{n+3}\right) \]
so $\mathcal{I}_\Gamma/\mathcal{I}_\Gamma^2 = (\vartheta_1')\oplus(\vartheta_3) = (s\vartheta_1)\oplus(s^n\vartheta_{n+3})$, and therefore $\Gamma$ is a $(-1,-n)$-curve.
\end{proof}

\subsection{The mirror of a $(-2,0)$-flop?}
Mirror symmetry has been well-studied in the context of the Atiyah flop over an ordinary 3-fold double point, e.g.\ \cite{mirror-flop}. In this case, the cones $\sigma\subset N_\RR$ and $\sigma^\star\subset M_\RR$ appearing in \S\ref{sect!atiyah-flop} are isomorphic, meaning that both $Y=\Spec\CC[\sigma^\star\cap M]$ and $Y^\star=\Spec\CC[\sigma\cap N]$ are ordinary toric 3-fold double points. A small resolution $f\colon X\to Y$ can be viewed as a deformation of the symplectic structure on $Y$, and under mirror symmetry this corresponds to a smoothing $W$ of $Y^\star$ given by a deformation of the complex structure on $Y^\star$, with vanishing cycle a Lagrangian sphere $S^3$. Indeed a 3-fold \emph{conifold transition} is the composition $X\to Y\rightsquigarrow Z$ of a small contraction to an ordinary node, followed by a smoothing of that node. Morrison conjectured that mirror symmetry reverses the order of a conifold transition, i.e.\ that we get a conifold transition of mirrors $Z^\star\to Y^\star\rightsquigarrow X^\star$. This conjecture will be proved in upcoming work of Ruddat \& Siebert \cite{rs} by using similar methods to this paper, involving integral affine manifolds and the Gross--Siebert program. 

Since the mirror $Y^\star$ that we construct for a $(-2,0)$-flop is not isomorphic to the original singularity $Y$, there is no analogue of a conifold transition in our setting. However, starting on $X$ and flopping twice $X\dashrightarrow X' \dashrightarrow X$ induces a `flop-flop' functor $F\colon \mathcal D(X)\to \mathcal D(X)$. In the case of the Atiyah flop, this is a special kind of autoequivalence of the derived category of $X$ called a spherical twist, which corresponds under mirror symmetry to an autoequivalence on the Fukaya category of $W$ induced by a Dehn twist around the vanishing sphere in the degeneration $W\rightsquigarrow Y^\star$. The flop-flop functor induced by a $(-2,0)$-flop has been studied by Toda \cite{toda}, but unfortunately is no longer a spherical twist. However, it would be interesting to know if it is still possible to interpret the mirror of this flop-flop functor in terms of a smoothing of the mirror singularity $Y^\star$ that we have constructed.

\end{document}